\documentclass[amssymb,amscd,11pt,20pt]{amsart}
\usepackage{amsmath,amsthm}
\usepackage{amsfonts}
\usepackage{amscd}
\usepackage{amscd}
\usepackage[all]{xy}

 \newtheorem{thm}{Theorem}[section]
 \newtheorem{cor}[thm]{Corollary}
 
 \newtheorem{prop}[thm]{Proposition}
 \theoremstyle{definition}
 \newtheorem{defn}[thm]{Definition}
 \theoremstyle{remark}
 \newtheorem{rem}[thm]{Remark}
 \newtheorem{ejem}[thm]{Example}
 \newtheorem{ejems}[thm]{Examples}

\numberwithin{equation}{subsection}
\newcommand{\OO}{{\mathcal O}}
\newcommand{\M}{{\mathcal M}}
\newcommand{\A}{{\mathcal A}}
\newcommand{\Sc}{{\mathfrak S}}
\newcommand{\U}{{\mathcal U}}
\newcommand{\Bc}{{\mathcal B}}
\newcommand{\HH}{{\mathcal H}}
\newcommand{\K}{{\mathcal K}}
\newcommand{\Nc}{{\mathcal N}}

\newcommand{\ZZ}{{\mathbb Z}}
\newcommand{\Lc}{{\mathcal L}}
\newcommand{\V}{{\mathcal V}}
\newcommand{\T}{{\mathcal T}}
\newcommand{\LL}{{\mathbb L}}
\newcommand{\RR}{{\mathbb R}}
\DeclareMathOperator{\ilim}{\underset\to\lim}
\DeclareMathOperator{\Aut}{{Aut}}
\DeclareMathOperator{\Spec}{{Spec}}
\DeclareMathOperator{\Coker}{{Coker}}
\newcommand{\pp}{{\mathfrak p}}
\newcommand{\wh}{\widehat}

\newcommand{\punto}{{\displaystyle \cdot}}
\newcommand{\proda}[2][]{\underset{#2}{\overset{#1}{\prod}}}
\newcommand{\suma}[2][]{\underset{#2}{\overset{#1}{\sum}}}

\newcommand{\enumera}{\begin{enumerate}}
\newcommand{\eenumera}{\end{enumerate}}
\newcommand{\C}{{\mathcal C}}
 
\DeclareMathOperator{\Hom}{{Hom}}
\DeclareMathOperator{\di}{{d}}
\DeclareMathOperator{\Id}{{Id}}

\begin{document}

\title[Ringed Finite Spaces]
 {Ringed Finite Spaces}

\author{ Fernando Sancho de Salas}
\address{Departamento de Matem\'{a}ticas and Instituto Universitario de Física Fundamental y Matemáticas (IUFFyM), Universidad de Salamanca,
Plaza de la Merced 1-4, 37008 Salamanca, Spain}

\email{fsancho@usal.es}

\subjclass[2010]{14XX, 55XX, 05XX, 06XX}

\keywords{Finite spaces, quasi-coherent modules}

\thanks {The  author was supported by research project MTM2009-07289 (MEC)}




\begin{abstract} A ringed finite space is a ringed space whose underlying topological space is finite. The category of ringed finite spaces contains, fully faithfully, the category of finite topological spaces and the category of affine schemes. Any ringed space, endowed with a finite open covering, produces a ringed finite space. We make a study of the homotopy of ringed finite spaces (that extends the homotopy of finite topological spaces) and a study of the quasi-coherent sheaves on a ringed finite space. This leads to set and develop basic notions on ringed finite spaces and morphisms, as being affine, schematic, semi-separated, etc, focusing on its cohomological properties. Finally, we see how to embed the category of quasi-compact and quasi-separated schemes in a localization of the category of finite ringed spaces.

\end{abstract}

\maketitle

\section*{Introduction}

A ringed finite space is a ringed space $(X,\OO_X)$ whose underlying topological space $X$ is finite, i.e. it is a finite topological space endowed with a sheaf of (commutative with unit) rings. It is well known (since Alexandroff) that a finite topological space is equivalent to a finite preordered set, i.e. giving a topology on a finite set is equivalent to giving a preorder relation. Giving a sheaf of rings $\OO_X$ on a finite topological space is equivalent to give, for each point $p\in X$, a ring $\OO_p$, and for each $p\leq q$ a morphism of rings $r_{pq}\colon \OO_p\to\OO_q$, satisfying the obvious relations ($r_{pp}=\Id$ for any $p$ and $r_{ql}\circ r_{pq}=r_{pl}$ for any $p\leq q\leq l$). The category of ringed finite spaces   is a full  subcategory of the category of ringed spaces and it contains (fully faithfully) the category of finite topological spaces (that we shall refer to as ``the topological case'') and the category of affine schemes (see Examples \ref{ejemplos}, (1) and (2)). If $(S,\OO_S)$ is an arbitrary ringed space (a topological space, a differentiable manifold, a scheme, etc) and we take a finite covering $\U=\{ U_1,\dots,U_n\}$ by open subsets, there is a natural associated ringed finite space $(X,\OO_X)$ and a morphism of ringed spaces $S\to X$ (see Examples \ref{ejemplos}, (3)). A natural question arise: which properties or structures of $S$ are determined by $X$? For example: do $S$ and $X$ have the same fundamental group, or the same category of quasi-coherent modules?  A particular interesting case is when $S$ is a quasi-compact and quasi-separated scheme and $\U$ is a (locally affine) finite covering by affine open subschemes; this case will be refered to as the ``algebro-goemetric case''. One of the aims of this paper is to show that the category of ringed finite spaces is a good framework where topological methods (those of finite topological spaces), from one side, and algebro-geometric ones, from the other side, are connected.

In any ringed space $(S,\OO_S)$ one has the notions of $\OO_S$-modules (sheaves of $\OO_S$-modules), quasi-coherent $\OO_S$-modules, coherent modules, etc. In section 2 we make an elementary study of quasi-coherent modules on a ringed finite space $(X,\OO_X)$. The main result is Theorem \ref{qc}. In the topological case, a quasi-coherent module is a locally constant sheaf (of abelian groups), i.e. a representation of the fundamental group in an abelian group; in the algebro-geometric case (i.e. $(S,\OO_S)$ is a quasi-compact and quasi-separated scheme and $\pi\colon S\to X$ is the morphism to the finite ringed space associated to a (locally affine) finite covering of $S$), the inverse image functor $\pi^*$ yields an equivalence between the category of quasi-coherent modules on $X$ and the category of quasi-coherent sheaves on $S$ (see Theorem \ref{schemes} and its topological analog Theorem \ref{fin-sp-assoc-top}). In \cite{EstradaEnochs} it is proved that the category of quasi-coherent sheaves on a quasi-compact and quasi-separated scheme $S$ is equivalent to the category of quasi-coherent $R$-modules, where $R$ is a ring representation of a finite quiver $\V$. Our point of view is that the quiver $\V$ may be thought of as a finite topological space $X$ and the representation $R$ as a sheaf of rings $\OO_X$. The advantage is that the equivalence between quasi-coherent modules is obtained from a geometric morphism $\pi\colon S\to X$. For example, we prove that this equivalence preserves cohomology: the cohomology of a quasi-coherent sheaf on $S$ coincides with the cohomology of the associated quasi-coherent sheaf on $X$ (see Example \ref{coh-schemes} and its topological analog).

In section 3 we make an elementary study of the homotopy of ringed finite spaces. We see how the homotopy relation of continous maps between finite topological spaces can be generalized to morphisms between ringed finite spaces in such a way that Stone's classification (\cite{Stong}) of finite topological  spaces  (via minimal topological spaces)  can be generalized to ringed finite spaces (Theorem \ref{homotopic-classification}). An important fact is that the category of quasi-coherent sheaves on a ringed finite space is an homotopical invariant: two homotopically equivalent ringed finite spaces have equivalent categories of quasi-coherent sheaves (Theorem \ref{homotinvariance}).

A great part of the paper deals with a particular case of ringed finite spaces: By a {\it finite space} we mean a ringed finite space $(X,\OO_X)$ such that the morphisms $r_{pq}\colon \OO_p\to\OO_q$ are flat. Let us justify the importance of this condition. Under this flatness assumption, the category of quasi-coherent modules on $X$ is an abelian subcategory of the abelian category of all $\OO_X$-modules. If $\OO_X$ is a sheaf of noetherian rings (i.e., $\OO_p$ is noetherian for any $p\in X$), then this flatness condition is equivalent to say that the structure sheaf $\OO_X$ is coherent. From the point of view of integral functors, the flatness assumption allows to define integral functors between the derived categories of quasi-coherent sheaves (see Corollary \ref{integral-functors}). Finally, our main examples (i.e. the ``topological case'' and the ``algebro-geometric case'') satisfy this flatness condition.

Section 4 is devoted to the study of the main cohomological properties of quasi-coherent sheaves on  finite spaces. The main results are Theorem \ref{qc-of-proj}, that studies the behaviour of quasi-coherence under projections, and Theorem  \ref{graphic}, that studies the cohomological structure of the graphic of a morphism; these results are essential for the rest of the paper.

In section 5 we make an elementary study of {\it affine} finite spaces. While in section 2 (Homotopy) the topological case was the guide, now is the algebro-geometric case. Our notion of affine space (resp. quasi-affine, Serre-affine space) is inspired in the algebro-geometric case, i.e., in the characterization of an affine scheme by its quasi-coherent modules. So it is a  natural notion for an algebro-geometer. In the topological case, a finite space is affine if and only if it is homotopically trivial. Hence the concept of affine finite space unifies the concepts of an affine scheme and that of a homotopically trivial finite topological space. The relative notion of affineness is also developed, i.e. the notion of an affine morphism.

Sections 6 and 7 are devoted to the study of those finite spaces and morphisms which have a good behavior with respect to quasi-coherent sheaves, where ``good'' means that they share the most relevant properties of quasi-coherent sheaves on (quasi-compact and quasi-separated) schemes. Let $(X,\OO_X)$ be a finite space and $\delta\colon X\to X\times X$ the diagonal morphism.  We say that $X$ is {\it schematic} if $\RR\delta_*\OO_X$ is quasi-coherent (i.e. $R^i\delta_*\OO_X$ is quasi-coherent for any $i$). The name is due to the fact that the ''algebro-goemetric'' case is the main example of a schematic finite space. The schematic condition is equivalent to the following property: for any open subset $j\colon U\hookrightarrow X$, $\RR j_*\OO_{X\vert U}$ is quasi-coherent, i.e. the open inclusions $j\colon U\hookrightarrow X$ have a good behavior with respect quasi-coherence. It turns out that, if $X$ is schematic, then $\RR j_*\Nc$ is quasi-coherent for any quasi-coherent module $\Nc$ on $U$. In particular, quasi-coherent modules have the extension property (see Theorem \ref{extension}).
A more restrictive notion is that of a semi-separated finite space (which is the analog of a semi-separated scheme). They are defined as those schematic spaces such that $R^i\delta_*\OO_X=0$ for $i>0$. We prove that this is equivalent to say that the diagonal morphism $\delta$ is affine. As it happens with schemes, being schematic is a local question  (but being semi-separated is not) and every schematic affine space is semi-separated.  Section 7 is devoted to schematic morphisms: Let $f\colon X\to Y$ be a morphism between finite spaces and $\Gamma\colon X\to X\times Y$ its graphic. We say that $f$ is schematic if $\RR\Gamma_*\OO_X$ is quasi-coherent. We prove that, if $f$ is schematic, then $\RR f_*\M$ is quasi-coherent for any quasi-coherent module $\M$ on $X$, and the converse is also true if $X$ is schematic (Theorem \ref{sch-preserv-qc2}). The local structure of schematic spaces and morphisms, their behavior  under direct products or compositions,  their structure and properties in the affine case, Stein's factorization and other questions are also treated in sections 6 and 7. To end with these sections we see an important fact: though the category of ringed finite spaces has fibered products,  the category of finite spaces has not; however, we prove that the category of schematic spaces and schematic morphisms has fibered products (Theorem \ref{fiberedproduct}).

Section 8 deals with the problem of embedding, fully faithfully, the category of quasi-compact and quasi-separated schemes into the category of schematic spaces and schematic morphisms. To each  quasi-compact and quasi-separated scheme, we can associate   an schematic space by the election of a finite (locally affine) affine covering. One can recover the initial scheme by the well known gluing technique of schemes. Of course the associated finite space depends on the election of the covering, since the finite spaces associated to different coverings of the same scheme are not isomorphic (not even homotopic). To avoid this problem (in other words, to identify the finite spaces associated to different coverings of the same scheme), we have to localize the category of schematic spaces by a certain class of morphisms which we have called qc-isomorphisms. A qc-isomorphism is an schematic and affine morphism $f\colon X\to Y$ such that $f_*\OO_X=\OO_Y$. The main result is Theorem \ref{embedding-schemes}, that states that there is a fully faithful functor from the category of quasi-compact and quasi-separated schemes to the localization of the category of schematic spaces (and schematic morphisms) by qc-isomorphisms. Moreover, this functor yields an equivalence between the category of affine schemes and the (localized) category of affine schematic spaces (Theorem \ref{affinescheme-affinespace}). The main ingredient is Grothendieck's faithfully flat descent (\cite{Grothendieck}), for the proof of Proposition \ref{S(affine)}.

Many of the results and techniques of this paper are  generalizable to Alexandroff spaces (those topological spaces where each point has a minimal open subset containing it), or finite quivers. Instead of dealing with the greatest possible generality, we have preferred to restrict ourselves to finite spaces, as a guiding and fruitful model for other more general situations.

This paper is dedicated to the beloved memory of Prof. Juan Bautista Sancho Guimer{\'a}. I learned from him most of mathematics I know, in particular the use of finite topological spaces in algebraic geometry.

\section{Preliminaries}

In this section we recall elementary facts about finite topological spaces and ringed spaces. The reader may consult \cite{Barmak} for the results on finite topological spaces and \cite{GrothendieckDieudonne} for  ringed spaces.

\subsection{Finite topological spaces}
\medskip

\begin{defn} A finite topological space  is a topological space with a finite number of points.
\end{defn}

Let $X$ be a finite topological space. For each $p\in X$, we shall denote by  $U_p$  the minimum open subset containing  $p$, i.e., the intersection of all the open subsets containing $p$. These $U_p$ form  a minimal base of open subsets.

\begin{defn} A finite preordered set is a finite set with a reflexive and transitive relation (denoted by $\leq$).
\end{defn}

\begin{thm} There is an equivalence between finite topological spaces and finite preordered sets.

\end{thm}

\begin{proof} If $X$ is a finite topological space, we define the relation: $$p\leq q\quad\text{iff}\quad p\in \bar q \quad (\text{i.e., if } q\in U_p) $$
Conversely, if $X$ is a finite preordered set, we define the following topology on  $X$: the closure of a point $p$ is $\bar p=\{ q\in X: q\leq p\}$.
\end{proof}

\begin{rem} \begin{enumerate}
\item The preorder relation defined above does not coincide with that of \cite{Barmak}, by with its inverse. In other words, the topology associated to a preorder that we have defined above is the dual topology that the one considered in op.cit.
\item If $X$ is a finite topological space, then $U_p=\{ q\in X: p\leq q\}$. Hence $X$ has a minimum $p$ if and only if $X=U_p$.
\end{enumerate}
\end{rem}

A map  $f\colon X\to X'$ between finite topological spaces is continuous if and only if it is monotone: for any $p\leq q$, $f(p)\leq f(q)$.

\begin{prop} A finite topological space is  $T_0$ (i.e., different points have different closures) if and only if the relation $\leq$ is antisimetric, i.e., $X$ is a partially ordered finite set (a finite poset).
\end{prop}

Any topological space $S$ has an associated $T_0$-topological space, $T_0(S):=$  the quotient of $S$ by the equivalence relation $p\sim q$ iff $\bar p = \bar{q}$. The quotient map $S\to T_0(S)$  is universal for morphisms from $S$ to $T_0$-spaces.

\medskip
\noindent{\it\ \ Dimension}. The dimension of a finite topological space is the maximum of the lengths of the chains of irreducible closed subsets. Equivalently, it is the maximum of the lengths of the  chains of points  $x_0<x_1<\cdots <x_n$.

\medskip
\begin{ejem}\label{covering}{\bf (Finite topological space associated to a finite covering)}. Let $S$ be a topological space and let   $\U=\{U_1,\dots,U_n\}$ be a finite open covering of $S$. Let us denote by $S_\U$ the set $S$ with the (finite) topology generated by $\U$ (i.e., an open subset is a finite union of finite intersections of the $U_i's$). For each $s\in S$, let us denote by $U_s$ the intersection of the $U_i$ containing $s$. These $U_s$ are a basis of the topological space $S_\U$.     Let $X=T_0(S_\U)$ be the associated $T_0$-space. This is a finite $T_0$-topological space, and we have a continuous map $\pi\colon S\to X$, $s\mapsto [s]$. For each $p=[s]\in X$, one has that $\pi^{-1}(U_{[s]})=U_s$. We shall say that $X$ is the  finite topological space associated to the topological space $S$ and the finite covering $\U$.

This construction is functorial in $\U$: Let $\U'=\{ U'_j\}$ be another  finite covering of $S$ and $\pi'\colon S\to X'$ its associated finite topological space. If $\U'$ is thinner than $\U$ (i.e., for each $s\in S$, $U'_s\subseteq U_s$), then one has a continuous map $X'\to X$ and a commutative diagram $$\xymatrix{ & S\ar[rd]^\pi \ar[ld]_{\pi'} & \\ X'\ar[rr] & & X.}$$

Functoriality on $S$: Let $f\colon S'\to S$ a continuous map, $\U$ a finite covering of $S$ and $\U'$ a finite covering of $S'$ that is thinner than $f^{-1}(\U)$. If $\pi\colon S\to X$ and $\pi'\colon S'\to X$ are the associated finite spaces, one has a continuous map $X'\to X$ and a commutative diagram
\[\xymatrix{ S'\ar[r]^f\ar[d]_{\pi'} & S\ar[d]^\pi\\ X'\ar[r] & X.
}\]

A more intrinsic construction of the finite topological space associated to a finite covering is given by spectral methods (see \cite{TeresaSancho}); indeed, let $\T_\U$ be the topology generated by $\U$ and $\T_S$ the topology of $S$. Then $X=\Spec\T_\U$ and the morphism $S\to X$ corresponds to the inclusion $\T_\U\hookrightarrow \T_S$. In general, for any distributive lattice $B$ and any topological space $S$ there is a bijective correspondence between morphisms (of distributive lattices) $B\to\T_S$ and continuous maps $S\to\Spec B$. In fact, given a continuous map $f\colon S\to\Spec B$, it induces a morphism $f^{-1}\colon \T_{\Spec B}\to \T_S$ whose composition with the natural inclusion $B\hookrightarrow \T_{\Spec B}$, gives the correspondent morphism $B\to\T_S$. For any finite topological space $X$ one has a canonical homeomorphism $T_0(X)=\Spec\T_X$

These spectral methods are only used in the proof of Theorem \ref{embedding-schemes}.

\end{ejem}

%

\subsection{Generalities on ringed spaces}
\medskip

\begin{defn} A ringed space is a pair $(X,\OO)$, where $X$ is a topological space and  $\OO$ is a sheaf of (commutative with unit) rings on  $X$. A morphism or ringed spaces $(X,\OO)\to (X',\OO')$ is a pair $(f,f^\#)$, where $f\colon X\to X'$ is a continuous map and $f^\#\colon \OO'\to f_*\OO'$ is a morphism of sheaves of rings (equivalently, a morphism of sheaves of rings $f^{-1}\OO'\to \OO$).
\end{defn}

\begin{defn} Let $\M$ be an $\OO$-module (a sheaf of $\OO$-modules). We say that $\M$ is {\it quasi-coherent} if for each $x\in X$ there exist an open neighborhood  $U$ of $x$ and an exact sequence
\[ \OO_{\vert U}^I \to \OO_{\vert U}^J\to\M_{\vert U}\to 0\] with $I,J$ arbitrary sets of indexes. Briefly speaking, $\M$ is quasi-coherent if it is locally a cokernel of free modules. We say that $\M$ is an $\OO$-module {\it of finite type} if, for each  $x\in X$, there exist an open neighborhood $U$ and an epimorphism
\[ \OO_{\vert U}^n\to\M_{\vert U}\to 0,\] i.e., $\M$ is locally a quotient of a finite free module. We say that $\M$ is an  $\OO$-module {\it of finite presentation} if, for each point $x\in X$, there exist an open neighborhood  $U$ and an exact sequence
\[ \OO_{\vert U}^m\to \OO_{\vert U}^n\to\M_{\vert U}\to 0.\] That is, $\M$ is locally a cokernel of finite free modules. Finally, we say that    $\M$ is {\it coherent} if it is of finite type and for any open subset $U$ and any morphism $\OO_{\vert U}^n\to\M_{\vert U}$, the kernel is an $\OO_{\vert U}$-module of finite type. In other words, $\M$ is coherent if it is of finite type and for every open subset $U$  any submodule of finite type of   $\M_{\vert U}$ is of finite presentation.
\end{defn}

Let $f\colon X\to Y$ a morphism of ringed spaces. If $\M$ is a quasi-coherent (resp. of finite type) module  on $Y$, then $f^*\M$ is a quasi-coherent (resp. of finite type) module on $X$.

Let $f\colon \M\to\Nc$ be a morphism of $\OO$-modules. If $\M$ and $\Nc$ are quasi-coherent, the cokernel $\Coker f$ is quasi-coherent too, but the kernel may fail to be quasi-coherent.

Direct sums and direct limits of quasi-coherent modules are quasi-coherent. The tensor product of two quasi-coherent modules is also quasi-coherent.

\section{Ringed finite spaces}

Let $X$ be a finite topological space. Recall that we have  a preorder relation \[ p\leq q \Leftrightarrow p\in \bar q \Leftrightarrow U_q\subseteq U_p\]

 Giving a sheaf $F$ of abelian groups (resp. rings, etc) on $X$ is equivalent to giving the following data:

 - An abelian group  (resp. a ring, etc) $F_p$ for each $p\in X$.

 - A morphism of groups (resp. rings, etc) $r_{pq}\colon F_p\to F_q$ for each $p\leq q$, satisfying: $r_{pp}=\Id$ for any $p$, and $r_{qr}\circ r_{pq}=r_{pr}$ for any $p\leq q\leq r$. These $r_{pq}$ are called {\it restriction morphisms}.

Indeed, if $F$ is a sheaf on $X$, then  $F_p$ is the stalk of $F$ at $p$, and it coincides with the sections of $F$ on $U_p$. That is
\[ F_p=\text{ stalk of } F \text{ at } p = \text{ sections of } F \text{ on } U_p:=F(U_p)\]
The morphisms $F_p\to F_q$ are just the restriction morphisms  $F(U_p)\to F(U_q)$.

\begin{ejem} Given a group $G$, the constant sheaf $G$ on $X$ is given by the data: $G_p=G$ for any $p\in X$, and $r_{pq}=\Id$ for any $p\leq q$.
\end{ejem}

\begin{defn} A {\it ringed finite space} is a ringed space $(X,\OO )$ such that  $X$ is a finite topological space.
\end{defn}

By the previous consideration, one has a ring $\OO_p$ for each $p\in X$, and a morphism of rings $r_{pq}\colon \OO_p\to\OO_q$ for each $p\leq q$, such that $r_{pp}=\Id$ for any $p\in X$ and  $r_{ql}\circ r_{pq}=r_{pl}$ for any $p\leq q\leq l$.
\medskip

 Giving a morphism of ringed spaces $(X,\OO)\to (X',\OO')$ between two ringed finite spaces, is equivalent to giving:

-  a continuous (i.e. monotone) map $f\colon X\to X'$,

 -  for each  $p\in X$, a ring homomorphism  $f^\#_p\colon \OO'_{f(p)}\to \OO_p$, such that, for any  $p\leq q$, the diagram (denote $p' =f(p), q'=f(q)$)
\[ \xymatrix{ \OO'_{p'} \ar[r]^{f^\#_{p}} \ar[d]_{r_{p'q'}} & \OO_{p}\ar[d]^{r_{pq}}\\ \OO'_{q'} \ar[r]^{f^\#_{q}}   & \OO_{q}}\] is commutative. We shall denote by $\Hom(X,Y)$ the set of morphisms of ringed spaces between two ringed spaces $X$ and $Y$.

\begin{ejems}\label{ejemplos} \item[$\,\,$(1)] {\it Punctual ringed spaces}. A ringed finite space is called punctual if the underlying topological space has only one element.  The sheaf of rings is then just a ring. We shall denote by  $(*,A)$ the ringed finite space with topological space $\{*\}$ and ring $A$. Giving a morphism of ringed spaces  $(X,\OO)\to (*,A)$ is equivalent to giving a ring homomorphism $A\to \OO(X)$. In particular, the category of punctual ringed spaces is equivalent to the (dual) category of rings, i.e., the category of affine schemes. In other words, the category of affine schemes is a full subcategory of the category of ringed finite spaces, precisely the full subcategory of punctual ringed finite spaces.

Any ringed space $(X,\OO)$ has an associated punctual ringed space $(*,\OO(X))$ and a morphism or ringed spaces $\pi\colon (X,\OO)\to (*,\OO(X))$ which is universal for morphisms from $(X,\OO)$ to punctual spaces. In other words, the inclusion functor
\[i\colon \{\text{Punctual ringed spaces}\} \hookrightarrow \{\text{Ringed spaces}\}\] has a left adjoint: $(X,\OO)\mapsto (*,\OO(X))$. For any $\OO(X)$-module $M$, $\pi^*M$ is a quasi-coherent module on $X$. We sometimes denote $\widetilde M:=\pi^*M$. If $X\to Y$ is a morphism of ringed spaces and $A\to B$ is the induced morphism between the global sections of $\OO_Y$ and $\OO_X$, then, for any $A$-module $M$ one has that $f^*\widetilde M = \widetilde{M\otimes_AB}$.
\medskip

\item[$\,\,$(2)] {\it Finite topological spaces}. Any finite topological space $X$ may be considered as a ringed finite space, taking
the constant sheaf $\ZZ$ as the sheaf of rings. If  $X$ and $Y$ are two finite topological spaces, then giving a morphism of ringed spaces  $(X,\ZZ)\to (Y,\ZZ)$ is just giving a continuous map $X\to Y$. Therefore the category of finite topological spaces is a full subcategory of the category of ringed finite spaces. The (fully faithful) inclusion functor
\[ \aligned \{\text{Finite topological spaces}\} &\hookrightarrow \{\text{Ringed finite spaces} \}\\ X &\mapsto (X,\ZZ)\endaligned\] has a left adjoint, that maps a ringed finite space $(X,\OO)$ to $X$. Of course, this can be done more generally, removing the finiteness hypothesis: the category of topological spaces is a full subcategory of the category of ringed spaces (sending $X$ to $(X,\ZZ)$), and this inclusion has a left adjoint: $(X,\OO)\mapsto X$.
\medskip

\item[$\,\,$(3)] Let $(S,\OO_S)$ be a ringed space (a scheme, a differentiable manifold, an analytic space, ...).
Let $\U=\{U_1,\dots,U_n\}$ be a finite open covering of $S$. Let $X$ be the finite topological space associated to $S$ and $\U$, and $\pi\colon S\to X$ the natural continuous map  (Example \ref{covering}). We have then a sheaf of rings on $X$, namely  $\OO:=\pi_*\OO_S$, so that $\pi\colon (S,\OO_S)\to (X,\OO)$ is a morphism of ringed spaces. We shall say that $(X,\OO)$ is the  {\it ringed finite space associated to the ringed space $S$ and the finite covering $\U$}. This construction is functorial on $\U$ and on $S$, as in Example \ref{covering}.

\medskip

\item[$\,\,$(4)] Let $(S,\OO_S)$ be a quasi-compact and quasi-separated scheme. It is not difficult to prove that  one can find an
affine covering $\U=\{U_1,\dots,U_n\}$ such that, for any $s\in S$, the intersection $U_s = \underset{s\in U_i}\cap U_i$ is affine (we say that $\U$ is locally affine). If $X$ is the finite topological space associated to $S$ and $\U$, the continuous map $\pi\colon S\to X$ satisfies that $\pi^{-1}(U_x)$ is affine for any $x\in X$. Conversely, if $S$ is a scheme and there exist a finite topological space $X$ and a continuous map $f\colon S\to X$ such that $f^{-1}(U_x)$ is affine for any $x\in X$, then $S$ is quasi-compact and quasi-separated.

\medskip
\item[$\,\,$(5)] Let $(X,\OO)$ be a ringed finite space. For each $p\in X$ let us denote $S_p$ the affine scheme $S_p=\Spec\OO_p$. For
each $p\leq q$, we have a morphism of schemes $S_q\to S_p$, induced by the ring homomorphism $\OO_p\to \OO_q$. We shall define
\[ \Sc (X):=\underset{p\in X}\ilim S_p\] where $\ilim$ is the direct limit (in the category of ringed spaces). More precisely: for each $p\leq q$,
let us denote $S_{pq}=S_q$. We have morphisms $S_{pq}\to S_q$ (the identity) and $S_{pq}\to S_p$ (taking spectra in the morphism $r_{pq}\colon \OO_p\to \OO_q$). We have then morphisms
\[ \underset{p\leq q}\coprod S_{pq}  \aligned \,_{\longrightarrow} \\ \,^{\longrightarrow} \endaligned \, \underset{p}\coprod S_p\] and we define the ringed space $\Sc (X)$ as the cokernel. That is, $\Sc (X)$ is the cokernel topological space, and $\OO_{\Sc (X)}$ is the sheaf of rings defined by: for any open subset $V$ of $\Sc (X)$, we define $\OO_{\Sc (X)}(V)$ as the kernel of
\[ \underset p\prod \OO_{S_p}(V_p)   \aligned \,_{\longrightarrow} \\ \,^{\longrightarrow} \endaligned \, \underset{p\leq q}\prod \OO_{S_{pq}}(V_{pq})\] where $V_p$ (resp. $V_{pq}$) is the preimage of $V$ under the natural map $S_p\to \Sc(X)$ (resp. $S_{pq}\to \Sc (X)$).
 In particular, $X$ and $\Sc (X)$ have the same global functions, i.e., $\OO_{\Sc (X)}(\Sc (X))=\OO_X(X)$. We say that $\Sc (X)$ is {\it the ringed space obtained by gluing the affine schemes $S_p$ along the schemes $S_{pq}$.} By definition of a cokernel (or a direct limit), for any ringed space $(T,\OO_T)$, the sequence
\[ \Hom(\Sc (X),T)\to \underset {p\in X}\prod \Hom(S_p,T)\aligned \,_{\longrightarrow} \\ \,^{\longrightarrow} \endaligned \, \underset{p\leq q}\prod \Hom (S_{pq},T)\] is exact.

If $X$ has a minimum $p$, the natural morphism $S_p\to \Sc (X)$ is an isomorphism. If, for any $p\leq q$, the morphism $S_q\to S_p$  is an open immersion, then $\Sc (X)$ is a (quasi-compact and quasi-separated) scheme. If $X$ is the ringed finite space associated to a (locally affine) finite covering $\U$ of a quasi-compact and quasi-separated scheme $S$, then $\Sc (X)=S$.

This construction is functorial: if $f\colon X'\to X$ is a morphism between ringed finite spaces, it induces a morphism $\Sc(f)\colon \Sc (X')\to \Sc (X)$. In particular,  for any ringed finite space $X$, the natural morphism $X\to (*,A)$ (with $ A=\OO(X)$) induces a morphism
\[ \Sc (X)\to \Spec A.\]
\end{ejems}

\subsection{Fibered products}

Let $X\to S$ and $Y\to S$ be morphisms between ringed finite spaces. The fibered product $X\times_SY$ is the ringed finite space whose underlying topological space is the ordinary fibered product of topological spaces (in other words it is the fibered product set with the preorder given by $(x,y)\leq (x',y')$ iff $x\leq x'$ and $y\leq y'$) and whose sheaf of rings is: if $(x,y)$ is an element of $X\times_SY$ and $s\in S$ is the image of $x$ and $y$ in $S$, then
\[\OO_{(x,y)}=\OO_x\otimes_{\OO_s}\OO_y\] and the morphisms $\OO_{(x,y)}\to \OO_{(x',y')}$ for each $(x,y)\leq (x',y')$ are the obvious ones. For any $(x,y)\in X\times_SY$, one has that $U_{(x,y)}=U_x\times_{U_s}U_y$, with $s$ the image of $x$ and $y$ in $S$.

One has natural morphisms $\pi_X\colon X\times_SY\to X$ and $\pi_Y\colon X\times_SY\to Y$, such that
\[ \aligned \Hom_S(T,X\times_SY)&\to \Hom_S(T,X)\times \Hom_S(T,Y)\\ f&\mapsto (\pi_X\circ f,\pi_Y\circ f)\endaligned\] is bijective.

When $S$ is a punctual space, $S=\{ *,k\}$, the fibered product will be donoted by $X\times_kY$ or simply by $X\times Y$ when $k$ is understood (or irrelevant).   The underlying topological space is the cartesian product $X\times Y$ and the sheaf of rings is given by $\OO_{(x,y)}=\OO_x\otimes_k\OO_y$.

If $f\colon X\to Y$ is a morphism of ringed finite spaces over $k$, the graphic $\Gamma_f\colon X\to X\times_kY$ is the morphism of ringed spaces corresponding to the pair of morphisms $\Id\colon X\to X$ and $f\colon X\to Y$.  Explicitly, it is  given by the continuous map $X\to X\times_kY$, $x\mapsto (x,f(x))$, and by the ring homomorphisms $\OO_x\otimes_k \OO_{f(x)}\to \OO_x$ induced by the identity $\OO_x\to\OO_x$ and by the morphisms $\OO_{f(x)}\to\OO_x$ associated with $f\colon X\to Y$.

More generally, if $X$ and $Y$ are ringed finite spaces over a ring finite space $S$ and $f\colon X\to Y$ is a morphism over $S$, the graphic of $f$ is the morphism $\Gamma_f\colon X\to X\times_SY$ corresponding to the pair of morphisms $\Id\colon X\to X$ and $f\colon X\to Y$.

\subsection{Quasi-coherent modules}

 Let $\M$ be a sheaf of  $\OO$-modules on a ringed finite space $(X,\OO)$. Thus, for each $p\in X$, $\M_p$ is an $\OO_p$-module and for each  $p\leq q$ one has a morphism of $\OO_p$-modules $\M_p\to\M_q$, hence a morphism of  $\OO_q$-modules
\[\M_p\otimes_{\OO_p}\OO_q\to\M_q\]

\begin{rem} From the natural isomorphisms \[\Hom_{\OO_{\vert U_p}}(\OO_{\vert U_p},\M_{\vert U_p})=\Gamma(U_p,\M)=\M_p=\Hom_{\OO_p}(\OO_p,\M_p)\] it follows that, in order to define a morphism of sheaves of modules $\OO_{\vert U_p}\to\M_{\vert U_p}$ it suffices to define a morphism of $\OO_p$-modules $\OO_p\to \M_p$ and this latter is obtained from the former by taking the stalk at $p$.
\end{rem}

\begin{thm}\label{qc} An $\OO$-module $\M$ is quasi-coherent if and only if for any  $p\leq q$ the morphism
\[\M_p\otimes_{\OO_p}\OO_q\to\M_q\]
is an isomorphism.
\end{thm}

\begin{proof} If $\M$ is quasi-coherent, for each point $p$ one has an exact sequence:
\[ \OO_{\vert U_p}^I\to \OO_{\vert U_p}^J \to \M_{\vert U_p} \to 0.\] Taking the stalk at $q\geq p$, one obtains an exact sequence
\[ \OO_q^I\to \OO_q^J \to \M_q \to 0\] On the other hand, tensoring the exact sequence at $p$  by $\otimes_{\OO_p}\OO_q$, yields an exact sequence
\[ \OO_q^I\to \OO_q^J \to \M_p\otimes_{\OO_p}\OO_q \to 0.\] Conclusion follows.

Assume now that  $\M_p\otimes_{\OO_p}\OO_q\to\M_q$ is an isomorphism for any $p\geq q$. We solve $\M_p$ by free $\OO_p$-modules:
\[ \OO_p^I\to \OO_p^J \to \M_p \to 0.\]
 We have then morphisms $\OO_{\vert U_p}^I\to \OO_{\vert U_p}^J \to \M_{\vert U_p}\to 0$. In order to see that this sequence is exact, it suffices to take the stalk at $q\geq p$. Now, the  sequence obtained at  $q$ coincides with the one obtained at $p$ (which is exact) after tensoring by $\otimes_{\OO_p}\OO_q$, hence it is exact.
\end{proof}

\begin{ejem} Let $(X,\OO)$ be a ringed finite space, $A=\OO(X)$ and $\pi\colon (X,\OO)\to (*,A)$ the natural morphism. We know that for any $A$-module $M$, $\widetilde M:=\pi^*M$ is a quasi-coherent module on $X$. The explicit stalkwise description of $\widetilde M$ is given by: $(\widetilde M)_x=M\otimes_A\OO_x$.
\end{ejem}

\begin{cor}\label{corqc} Let $X$ be a ringed finite space with a minimum and $A=\Gamma(X,\OO)$. Then the functors
\[\aligned \{\text{Quasi-coherent $\OO$-modules} \} & \overset{\longrightarrow}\leftarrow \{ \text{$A$-modules}\} \\ \M &\to \Gamma(X,\M) \\ \widetilde M &\leftarrow M \endaligned \]
are mutually inverse.
\end{cor}

\begin{proof} Let $p$ be the minimum of $X$. Then $U_p=X$ and for any sheaf $F$ on $X$, $F_p=\Gamma(X,F)$.
If $\M$ is a quasi-coherent module, then for any $x\in X$, $\M_x=\M_p\otimes_{\OO_p}\OO_x$. That is, $\M$ is univocally determined by its stalk at $p$, i.e., by its global sections.
\end{proof}

This corollary is a particular case of the invariance of the category of quasi-coherent mo\-du\-les under homotopies (see Theorem \ref{homotinvariance}), because any ringed finite space with a minimum $p$ is contractible to $p$ (Remark \ref{contractible}).

\begin{thm}  $\M$ is an $\OO$-module of finite type if and only if:

 - for each  $p\in X$, $\M_p$ is an $\OO_p$-module of finite type,

 - for any  $p\leq q$ the morphism \[\M_p\otimes_{\OO_p}\OO_q\to\M_q\] is surjective.
\end{thm}

\begin{proof} If $\M$ is of finite type, for each  $p$ one has an epimorphism  $\OO_{\vert U_p}^n\to\M_{\vert U_p}\to 0$. Taking, on the one hand,  the stalk at  $p$ tensored by  $\otimes_{\OO_p}\OO_q$, and on the other hand the stalk at $q$, one obtains a commutative diagram
\[\xymatrix {\OO_q^n \ar[r]\ar[d] & \M_p \otimes_{\OO_p}\OO_q \ar[r]\ar[d] & 0 \\ \OO_q^n \ar[r]  & \M_q  \ar[r]  & 0}\]  and one concludes. Conversely, assume that $\M_p$ is of finite type and  $\M_p\otimes_{\OO_p}\OO_q\to\M_q$ is surjective. One has an epimorphism $\OO_p^n\to\M_p\to 0$, that induces a morphism  $\OO_{\vert U_p}^n\to \M_{\vert U_p}$. This is an epimorphism because it is so at the stalk at any  $q\in U_p$.
\end{proof}

\begin{rem} Let $\M$ be an $\OO$-module on a ringed finite space. Arguing as in the latter theorem, one proves that  $ \M_p\otimes_{\OO_p}\OO_q\to\M_q$ is surjective for any $p\leq q$ if and only if $\M$ is locally a quotient of a free module (i.e., for each $p\in X$ there exist an open neighborhood $U$ of $p$ and an epimorphism $\OO_{\vert U}^I\to \M_{\vert U}\to 0$, for some  set of indexes $I$).
\end{rem}

\begin{thm}\label{coherent} An $\OO$-module $\M$ is coherent if and only if it is of finite type and it satisfies:

For each $p$, every sub-$\OO_p$-module of finite type $N$ of $\M_p$ is of finite presentation and, for any $q\geq p$, the natural morphism  $N\otimes_{\OO_p}\OO_q\to M_q$ is injective.

\end{thm}

\begin{proof} Let $\M$ be a coherent module. By definition, it is of finite type. Let $N$ be a submodule of finite type of  $\M_p$.  $N$ is the image of a morphism  $\OO_p^n\to \M_p$. This defines a morphism  $\OO_{\vert U_p}^n\to\M_{\vert U_p}$, whose kernel  $\K$ is of finite type because  $\M$ is coherent. Taking the the stalk at  $p$ one concludes that $N$ is of finite presentation. Moreover one has an exact sequence  $0\to \K_p\to\OO_p^n\to N\to 0$ and for any $q\geq p$ an exact sequence  $  \K_p \otimes_{\OO_p}\OO_q\to\OO_q^n\to N \otimes_{\OO_p}\OO_q\to 0$ and a commutative diagram
\[ \xymatrix{ & \K_p\otimes_{\OO_p}\OO_q \ar[r]\ar[d] & \OO_q^n \ar[r]\ar[d] & N \otimes_{\OO_p}\OO_q \ar[r]\ar[d] & 0
\\ 0\ar[r] &\K_q \ar[r]  & \OO_q^n \ar[r]  & \M_q   &  }  \] The surjectivity of $K_p\otimes_{\OO_p}\OO_q\to \K_q$ implies the injectivity of  $N\otimes_{\OO_p}\OO_q \to\M_q$.

Assume now that $\M$ is a module of finite type satisfying  the conditions. Let $U$ be an open subset and  $\OO_{\vert U}^n\to\M_{\vert U}$ a morphism, whose kernel is denoted by  $\K$. We have to prove that  $\K$ is of finite type. For each $p\in U$, the image, $N$, of $\OO_p^n\to\M_p$ is of finite presentation, hence  $\K_p$ is of finite type. For each $q\geq p$ we have a commutative diagram
\[ \xymatrix{ & \K_p\otimes_{\OO_p}\OO_q \ar[r]\ar[d] & \OO_q^n \ar[r]\ar[d] & N \otimes_{\OO_p}\OO_q \ar[r]\ar[d] & 0
\\ 0\ar[r] &\K_q \ar[r]  & \OO_q^n \ar[r]  & \M_q   &  }  \] Now, the injectivity of $N\otimes_{\OO_p}\OO_q \to\M_q$ implies the injectivity of  $K_p\otimes_{\OO_p}\OO_q\to \K_q$. Hence $\K$ is of finite type and  $\M$ is coherent.
\end{proof}

\begin{thm} $\OO$ is coherent if and only if:
\begin{enumerate} \item For each $p$, any finitely generated ideal of  $\OO_p$ is of finite presentation.
\item For any $p\leq q$, the morphism $\OO_p\to\OO_q$ is flat.
\end{enumerate}
\end{thm}

\begin{proof} It is a consequence of the previous theorem and the ideal criterium of flatness.
\end{proof}

\begin{cor} Let $(X,\OO )$ be a ringed finite space of noetherian rings (i.e,  $\OO_p$ is a noetherian ring for any $p\in X$). Then $\OO$ is coherent if and only if for any $p\leq q$ the morphism $\OO_p\to\OO_q$ is flat.
\end{cor}

\subsection{Finite spaces}

\begin{defn} A {\it finite space}  is a ringed finite space $(X,\OO)$ such that for any $p\leq q$ the morphism $\OO_p\to\OO_q$ is flat.
\end{defn}

Any open subset of a finite space is a finite space. The product of two finite spaces is a finite space.

\begin{prop} Let $(X,\OO)$ be a  finite space. Then, the kernel of any morphism between quasi-coherent $\OO$-modules is also quasi-coherent. Moreover, if
\[ 0\to\M'\to\M\to\M''\to 0\] is an exact sequence of  $\OO$-modules and two of them are quasi-coherent, then the third is quasi-coherent too. In conclusion, the category of quasi-coherent $\OO$-modules on a finite space is an abelian subcategory of the category of $\OO$-modules.
\end{prop}

\begin{proof} It follows easily from Theorem \ref{qc} and the flatness assumption.
\end{proof}

\begin{ejems} \begin{enumerate} \item Any finite topological space $X$ is a  finite space (with $\OO=\ZZ$), since the restrictions morphisms are the identity.
\item If $X$ is the ringed finite space associated to a (locally affine) finite affine covering of a quasi-compact and quasi-separated scheme $S$ (see Examples \ref{ejemplos}.3 and \ref{ejemplos}.4), then $X$ is a finite space. This follows from the following fact: if $V\subset U$ is an inclusion between two affine open subsets, the restriction morphism $\OO_S(U)\to\OO_S(V)$ is flat.
\end{enumerate}
\end{ejems}

\begin{thm}\label{schemes} Let $S$ be a quasi-compact and quasi-separated scheme and   $\U=\{ U_1,\dots, U_n\}$ a (locally affine) finite  covering by affine open subschemes. Let $(X,\OO)$ be the finite space associated to $S$ and $\U$,   and $\pi\colon S\to X$ the natural morphism of ringed spaces (see Examples \ref{schemes}, (3) and (4)). One has:

1. For any quasi-coherent   $\OO_S$-module $\M$, $\pi_*\M$ is a quasi-coherent $\OO$-module.

2. The functors $\pi^*$ and $\pi_*$ establish an equivalence between the category of quasi-coherent  $\OO_S$-modules and the category of quasi-coherent   $\OO$-modules.

Moreover, for any open subset $U$ of $X$, the morphism $\pi^{-1}(U)\to U$ satisfies 1. and 2.
\end{thm}

\begin{proof} 1. We have to prove that $(\pi_*\M)\otimes_{\OO_p}\OO_q\to(\pi_*\M)_q$ is an isomorphism for any $p\leq q$. This is a consequence of the following fact: if $V\subset U$ are open and affine subsets of a scheme $S$ and $\M$ is a quasi-coherent module on $S$, the natural map $\M(U)\otimes_{\OO_S(U)}\OO_S(V)\to\M(V)$ is an isomorphism.

2. Let $\M$ be a quasi-coherent module on $S$. Let us see that the natural map $\pi^*\pi_*\M\to\M$ is an isomorphism. Taking the stalk at $s\in S$, one is reduced to the following fact: if $U$ is an affine open subset of  $S$, then for any $s\in U$ the natural map $\M(U)\otimes_{\OO_S(U)}\OO_{S,s}\to \M_s$ is an isomorphism.

To conclude 2., let $\Nc$ be a quasi-coherent module on $X$ and let us see that the natural map $\Nc\to\pi_*\pi^*\Nc$ is an isomorphism. Taking the stalk at $p\in X$, we have to prove that $\Nc_p\to (\pi^*\Nc)(U)$ is an isomorphism, with $U=\pi^{-1}(U_p)$. Notice that $U$ is an affine subset and $\OO_S(U)=\OO_p$. It suffices to prove that, for any $s\in U$, $\Nc_p\otimes_{\OO_p}\OO_{S,s}\to (\pi^*\Nc)(U)\otimes_{\OO_p}\OO_{S,s}$ is an isomorphism. Denoting $q=\pi(s)$, one has that $(\pi^*\Nc)(U)\otimes_{\OO_p}\OO_{S,s}= (\pi^*\Nc)_s=\Nc_q\otimes_{\OO_q}\OO_{S,s}$. Since $\Nc$ is quasi-coherent, $\Nc_q=\Nc_p\otimes_{\OO_p}\OO_q$. Conclusion follows.

Finally, these same proofs work for $\pi\colon \pi^{-1}(U)\to U$, for any open subset $U$ of $X$.
\end{proof}

\begin{thm}\label{qc-fts} Let $X$ be a finite topological space $(\OO=\ZZ$). A sheaf $\M$ of abelian groups on $X$ is quasi-coherent if and only if it is locally constant, i.e., for each $p\in X$, $\M_{\vert U_p}$ is (isomorphic to) a constant sheaf. If $X$ is connected, this means that there exists an abelian group  $G$ such that $\M_{\vert U_p}=G$ for every $p$. If $X$ is not connected, the latter holds in each connected component. Moreover $\M$ is coherent if and only if $G$ is a finitely generated abelian group.
\end{thm}

\begin{proof} Since $\OO$ is the constant sheaf $\ZZ$, the quasi-coherence condition $$``\M_p\otimes_{\OO_p}\OO_q\to\M_q \text{ is an isomorphism}"$$ is equivalent to say that the restriction morphisms $\M_p\to\M_q$ are isomorphisms, i.e., $\M_{\vert U_p}$ is isomorphic to a constant sheaf. The statement on the coherence is  a consequence of Theorem \ref{coherent} and the hypothesis ($\OO=\ZZ$).
\end{proof}

Now let us prove a topological analog of Theorem \ref{schemes}. First let us recall a basic result about locally constant sheaves and the fundamental group.

\medskip
\noindent{\it Locally constant sheaves and the fundamental group}.
\medskip

Let $S$ be a path connected, locally path connected and locally simply connected topological space and let $\pi_1(S)$ be its fundamental group. Then there is an equivalence between the category of locally constant sheaves on $S$ (with fibre type $G$, an abelian group) and the category of representations of $\pi_1(S)$ on $G$ (i.e., morphisms of groups $\pi_1(S)\to \Aut_{\ZZ-\text{mod.}} G$). In particular, $S$ is simply connected if and only if any locally constant sheaf (of abelian groups) on $S$ is constant.

We can make a variant of this result. Assume that $X$ is connected. First, replace the constant sheaf $\ZZ$ by a locally constant sheaf of rings $\A$, i.e., a sheaf of rings that is locally isomorphic to a constant sheaf $A$ (a commutative ring). This is equivalent to give a ring representation of $\pi_1(X)$ in $A$, i.e.  a morphism of groups $\rho\colon \pi_1(X)\to \Aut_{\text{rings}}A$. For each $\sigma\in\pi_1(X)$,  $a\in A$, one denotes $\sigma(a)=\rho(\sigma)(a)$.

\begin{defn} A $\rho$-representation of $\pi_1(X)$ is an $A$-module  $M$ endowed with a representation  $\phi\colon \pi_1(X)\to \Aut_{\text{groups}}M$, which is compatible with $\rho$, that is:
\[ \sigma (am)= \sigma(a)\sigma(m)\]
where $\sigma(m)=\phi(\sigma)(m)$, with $m\in M$.
\end{defn}

Then, a $\rho$-representation of $\pi_1(X)$ is equivalent to a quasi-coherent $\A$-module. If $\rho$ is the trivial representation, then a $\rho$-representation is just a representation of $A$-modules, i.e. a morphism of groups $\pi_1(X)\to \Aut_{A-\text{mod.}}M$, and it is equivalent to a quasi-coherent $A$-module on $X$ (here $A$ denotes the constant sheaf of rings $A$).
\medskip

Now, the topological analog of Theorem \ref{schemes} is:

\begin{thm}\label{fin-sp-assoc-top} Let $S$ be a path connected, locally path connected and locally simply connected topological space and  let $\U=\{ U_1,\dots,U_n\}$ be a (locally simply connected) finite covering of $S$, i.e., for each $s\in S$, the intersection $\underset{s\in U_i}\cap U_i$ is simply connected. Let $X$ be the associated finite topological space and $\pi\colon S\to X$ the natural continous map. Then

1. For any locally constant sheaf $\Lc$ on $S$, $\pi_*\Lc$ is a locally constant sheaf on $X$.

2. The functors $\pi^*$ and $\pi_*$ establish an equivalence between the category of locally constant sheaves on $S$ and the category of locally constant sheaves on   $X$. In other words, $\pi_1(S)\to\pi_1(X)$ is an isomorphism.
\end{thm}

\begin{proof} Let us recall that, on a simply connected space, every locally constant sheaf is constant. Let $x\leq x'$ in $X$, and put $x=\pi(s)$, $x'=\pi(s')$. Then $(\pi_*\Lc)_x \to(\pi_*\Lc)_{x'}$ is the restriction morphism $\Gamma(U_s,\Lc)\to \Gamma(U_{s'},\Lc)$, which is an isomorphism because $\Lc$ is a constant sheaf on $U_s$.

If $\Lc$ is a locally constant sheaf on $S$, the natural morphism $\pi^*\pi_*\Lc\to\Lc$ is an isomorphism, since taking fibre at $s\in S$ one obtains the morphism $\Gamma(U_s,\Lc)\to \Lc_s$, which is an isomorphism because $\Lc$ is a constant sheaf on $U_s$. Finally, if $\Nc$ is a locally constant sheaf on $X$, the natural map $\Nc\to \pi_*\pi^*\Nc$ is an isomorphism: taking fibre at a point $x=\pi(s)$ one obtains the morphism $\Nc_x\to \Gamma(U_s, \pi^*\Nc)$, which is an isomorphism because $\Nc$ is a constant sheaf on $U_x$ (and then $\pi^*\Nc$ is a constant sheaf on $U_s$).
\end{proof}

\begin{rem} The same theorem holds in a more general situation: Let $S$ and $T$ be  path connected, locally path connected and locally simply connected topological spaces, $f\colon S\to T$ a continuous map  such that there exists a basis like open (and connected) cover $\U$ of $T$ such that $f^{-1}(U)$ is connected and $\pi_1(f^{-1}(U))\to \pi_1(U)$ is an isomorphism for every $U\in \U$. Then $\pi_1(S)\to\pi_1(T)$ is an isomorphism. It is an analogue of McCord's theorem (see \cite{Barmak}, Theorem 1.4.2) for the fundamental group. See also \cite{Quillen}, Proposition 7.6.
\end{rem}

\section{Homotopy}

For this section, we shall follow the lines of \cite{Barmak}, section 1.3, generalizing them to the ringed case.

Let $X$, $Y$ be  finite topological spaces and $\Hom(X,Y)$ the (finite) set of continuous maps. This set has a preorder, the pointwise preorder:
\[ f\leq g \iff f(x)\leq g(x) \text{ for any } x\in X,\]
hence $\Hom(X,Y)$ is a finite topological space.

It is easy to prove that two continous maps $f,g\colon X\to Y$ are homotopic (denoted by $f\sim g$) if and only if they belong to the same connected component of $\Hom(X,Y)$. In other words, if we denote $f\equiv g$ if either $f\leq g$ or $f\geq g$, then  $f\sim g$ if and only if there exists a sequence
\[ f=f_0\equiv f_1 \equiv \cdots \equiv f_n=g,\qquad f_i\in\Hom(X,Y)\]

Assume now that $X$ and $Y$ are ringed finite spaces and $\Hom(X,Y)$ is the set of morphisms of ringed spaces. It is no longer a finite set, however we can define a preorder relation:

\begin{defn} Let $f,g\colon X\to Y$ be two morphisms of ringed spaces. We say that $f\leq g$ if:

(1) $f(x)\leq g(x)$ for any $x\in X$.

(2) For any $x\in X$ the triangle
\[ \xymatrix{\OO_{f(x)}\ar[rr]^{r_{f(x)g(x)}}\ar[rd]_{f^\#_x} & & \OO_{g(x)}\ar[ld]^{g^\#_x}\\ & \OO_x & }\] is commutative. We shall denote by $f\equiv g$ if either $f\leq g$ or $f\geq g$.
\end{defn}

\begin{rem}\label{rem} If $f(x)=g(x)$ for any $x\in X$ (i.e., $f$ and $g$ coincide as continous maps) and $f\leq g$, then $f=g$.
\end{rem}

We can define the homotopy relation by:

\begin{defn} Let $f,g\colon X\to Y$ be two morphisms of ringed spaces. We say that $f$ and $g$ are homotopic,  $f\sim g$, if there exists a sequence:
\[ f=f_0\equiv f_1 \equiv \cdots \equiv f_n=g,\qquad f_i\in\Hom(X,Y) \]
\end{defn}

We can then define the homotopy equivalence between ringed finite spaces:

\begin{defn} Two  ringed finite spaces $X$ and $Y$ are said to be homotopy equivalent, denoted by $X\sim Y$, if there exist morphisms
$f\colon X\to Y$ and $g\colon Y\to X$ such that $g\circ f \sim \Id_X$ and $f\circ g\sim \Id_Y$.
\end{defn}

\begin{rem}\label{contractible} Any ringed finite space $X$ with a minimum $p$ is contractible to $p$, i.e. it is homotopy equivalent to the punctual ringed space
$(p,\OO_p)$. Indeed, one has a natural morphism $i_p\colon (p,\OO_p)\to X$. On the other hand, since $p$ is the minimum, $X=U_p$ and $\OO_p=\Gamma(X,\OO_X)$, and we have the natural morphism (see Examples \ref{ejemplos}, (1)) $\pi\colon X\to (p,\OO_p)$. The composition $\pi\circ i_p$ is the identity and   $i_p\circ \pi \geq \Id_X$.
\end{rem}

Let $f,g\colon X\to Y$ be two morphisms of ringed spaces, $S$ a subespace of $X$. We leave the reader to define the notion of being homotopic relative to $S$ and hence the notion of a strong deformation retract.

\begin{prop} Let $f,g\colon X\to Y$ be two morphisms of ringed spaces. If $f\sim g$, then, for any quasi-coherent sheaf $\M$ on $Y$, one has $f^*\M=g^*\M$.
\end{prop}

\begin{proof} We may assume that $f\leq g$. Then, for any $x\in X$, $$(f^*\M)_x=\M_{f(x)}\otimes_{\OO_{f(x)}}\OO_x = \M_{f(x)}\otimes_{\OO_{f(x)}}\OO_{g(x)}\otimes_{\OO_{g(x)}}\OO_x = \M_{g(x)}\otimes_{\OO_{g(x)}}\OO_x =(g^*\M)_x$$ where the second equality is due to the hypothesis $f\leq g$ and the third one to the quasi-coherence of $\M$.
\end{proof}

The following theorem is now straightforward (and it generalizes Corollary \ref{corqc}):

\begin{thm}\label{homotinvariance} If $X$ and $Y$ are homotopically equivalent, then their categories of quasi-coherent modules are equivalent. In other words, the category of quasi-coherent modules on a ringed finite space is a homotopic invariant.
\end{thm}

\subsection{Homotopical classification: minimal spaces}

Here we see that Stone's homotopical classification of finite topological spaces via minimal topological spaces can be reproduced in the ringed context.

First of all, let us prove that any ringed finite space is homotopically equivalent to its $T_0$-associated space. Let $X$ be a ringed finite space, $X_0$ its associated $T_0$-space and $\pi\colon X\to X_0$ the quotient map. Let us denote $\OO_0=\pi_*\OO$. Then $(X,\OO)\to (X_0,\OO_0)$ is a morphism of ringed spaces.  The preimage $\pi^{-1}$ gives a bijection between the open subsets of $X_0$ and the open subsets of $X$. Hence, for any $x\in X$, $\OO_x={\OO_0}_{\pi(x)}$, and any section $s\colon X_0\to X$ of $\pi$ is continuous and a morphism of ringed spaces. The composition $\pi\circ s$ is the identity and the composition $s\circ \pi$ is homotopic to the identity, because $\OO_x=\OO_{s(\pi(x))}$. We have then proved:

\begin{prop}\label{sdr} $(X_0,\OO_0) \hookrightarrow (X,\OO_X)$ is a strong deformation retract.
\end{prop}

Let $X$ be a ringed finite $T_0$-space. Let us generalize the notions of up beat point and down beat point to the ringed case.

\begin{defn} A point $p\in X$ is called a {\it down beat point} if  $\bar p -\{ p\}$ has a maximum. A point $p$ is called an {\it up beat point} if $U_p- \{ p\}$ has a minimum $q$ and $r_{pq}\colon \OO_p\to\OO_q$ is an isomorphism. In any of these cases we say that $p$ is a {\it beat point} of $X$.
\end{defn}

\begin{prop}\label{beating} Let $X$ be a ringed finite $T_0$-space and $p\in X$ a beat point. Then $X- \{ p\}$ is a strong deformation retract of $X$.
\end{prop}

\begin{proof} Assume that $p$ is a down beat point and let $q$ be the maximum of $\bar p- \{ p\}$. Define the retraction $r\colon X\to X-\{ p\}$ by $r(p)=q$. It is clearly continuous (order preserving). It is a ringed morphism because one has the restriction morphism $\OO_q\to\OO_p$. If $i\colon X-\{ p\}\hookrightarrow X$ is the inclusion, then $i\circ r\leq \Id_X$ and we are done.

Assume now that $p$ is an up beat point and let $q$ be the minimum of $U_p- \{ p\}$. Define the retraction $r\colon X\to X-\{ p\}$ by $r(p)=q$. It is order preserving, hence continuous. By hypothesis the restriction morphism $\OO_p\to\OO_q$ is an isomorphism, so that $r$ is a morphism of ringed spaces. Finally, $i\circ r\geq \Id_X$ and we are done.
\end{proof}

\begin{defn}  A ringed finite $T_0$-space is a {\it minimal} ringed finite space if it has no beat points. A core of a ringed finite space $X$ is a strong deformation retract which is a minimal ringed finite space.
\end{defn}

By Propositions \ref{sdr} and \ref{beating} we deduce that every ringed finite space
has a core. Given a ringed finite space $X$, one can find a $T_0$-strong deformation
retract $X_0\subseteq X$ and then remove beat points one by one to obtain a minimal
ringed finite space. As in the topological case, the notable property about this construction is that in fact the
core of a ringed finite space is unique up to isomorphism, moreover: two ringed finite spaces are homotopy equivalent if and only if their cores are isomorphic.

\begin{thm}\label{minimal} Let $ X$ be a minimal ringed finite space. A map $f \colon X \to X$ is
homotopic to the identity if and only if $f = \Id_X$.
\end{thm}

\begin{proof}  We may suppose that $ f \leq \Id_X$ or $f \geq \Id_X$. Assume
$ f \leq \Id_X$. By Remark \ref{rem} it suffices to prove that $f(x)=x$ for any $x\in X$. On the contrary,  let $p\in X$ be minimal with the condition $f(x)\neq x$. Hence $f(p)<p$ and $f(x)=x$ for any $x<p$. Then $f(p)$ is the maximum of $\bar p-\{ p\}$, which contradicts that $X$ has no down beat points.

Assume now that $f\geq \Id_X$. Again, it suffices to prove that $f(x)=x$ for any $x\in X$. On the contrary, let  $p\in X$ be maximal with the condition $f(x)\neq x$. Then $f(p)>p$ and $f(x)=x$ for any $x>p$. Hence $q=f(p)$ is the minimum of $U_p-\{ p\}$. Moreover  $f$ is a morphism of ringed spaces, hence it gives a commutative diagram
\[ \xymatrix {\OO_q=\OO_{f(p)} \ar[r]^{\quad f^\#_p}\ar[d]_{\Id} &\OO_p \ar[d]^{r_{pq}}\\ \OO_q=\OO_{f(q)}\ar[r]^{\quad f^\#_q} & \OO_q.}\]  Moreover, since $f\geq \Id_X$, the triangles
\[ \xymatrix{\OO_{p}\ar[rr]^{r_{pq}}\ar[rd]_{\Id^\#_p} & & \OO_{q}\ar[ld]^{f^\#_p}\\ & \OO_p & }\quad
\xymatrix{\OO_{q}\ar[rr]^{r_{qq}}\ar[rd]_{\Id^\#_q} & & \OO_{q}\ar[ld]^{f^\#_q}\\ & \OO_q & }\] are commutative. One concludes that   $r_{pq}$ is an isomorphism and $p$ is an up beat point of $X$.
\end{proof}

\begin{thm}\label{homotopic-classification} (Classification Theorem). A homotopy equivalence between
minimal ringed finite spaces is an isomorphism. In particular the core of a ringed
finite space is unique up to isomorphism and two ringed finite spaces are homotopy
equivalent if and only if they have isomorphic cores.
\end{thm}

\begin{proof}  Let $f \colon X \to Y$ be a homotopy equivalence between minimal ringed finite spaces and let $g \colon Y \to X$ be a homotopy inverse. Then $gf = \Id_X$ and $fg = \Id_Y$ by
Theorem \ref{minimal}. Thus, f is an isomorphism. If $X_1$ and $X_2$ are two cores of
a ringed finite space $X$, then they are homotopy equivalent minimal ringed finite spaces, and therefore, isomorphic. Two ringed finite spaces $X$ and $Y$ have the same
homotopy type if and only if their cores are homotopy equivalent, but this is
the case only if they are isomorphic.
\end{proof}

\section{Cohomology}

Let  $X$ be a finite topological space and  $F$ a sheaf of abelian groups on  $X$.

\begin{prop}\label{aciclicity} If $X$ is a finite topological space with a minimum, then $H^i(X,F)=0$ for any sheaf $F$ and any $i>0$. In particular, for any finite topological space one has
\[ H^i(U_p,F)=0\]
for any $p\in X$, any sheaf $F$ and any $i>0$.
\end{prop}

\begin{proof} Let $p$ be the minimum of $X$. Then $U_p=X$ and, for any sheaf $F$, one has $\Gamma(X,F)=F_p$; thus, taking global sections is the same as taking the stalk at $p$, which is an exact functor.
\end{proof}

Let $f\colon X\to Y$ a continuous map between finite topological spaces and $F$ a sheaf on $X$. The i-th higher direct image $R^if_*F$ is the sheaf on $Y$ given by:
\[ [R^if_*F]_y=H^i(f^{-1}(U_y),F)\]

\begin{rem} Let $X,Y$ be two finite topological spaces and $\pi\colon X\times Y\to Y$ the natural projection. If $X$ has a minimum ($X=U_x$), then, for any sheaf $F$ on $X\times Y$, $R^i\pi_*F=0$ for $i>0$, since $(R^i\pi_*F)_y=H^i(U_x\times U_y,F)=0$ by Proposition \ref{aciclicity}. In particular, $H^i(X\times Y, F)=H^i(Y,\pi_*F)$.
\end{rem}

\subsection{Standard resolution} Let $F$ be a sheaf on a finite topological space $X$. We define $C^nF$ as the sheaf on $X$  whose sections on an open subset $U$ are
\[ (C^nF)(U)=\proda{U \ni x_0<\cdots <x_n } F_{x_n}\] and whose  restriction morphisms $(C^nF)(U)\to (C^nF)(V)$ for any $V\subseteq U$ are the natural projections.

One has morphisms $d\colon C^nF \to C^{n+1}F$ defined in each open subset $U$ by the formula
\[ (\di a) (x_0<\cdots < x_{n+1})= \suma{0\leq i\leq n} (-1)^i a(x_0<\cdots \wh{x_i}\cdots <x_{n+1}) + (-1)^{n+1} \bar a (x_0<\cdots <x_n)   \] where $\bar a (x_0<\cdots <x_n)$ denotes the image of  $ a (x_0<\cdots <x_n)$ under the morphism $F_{x_n}\to F_{x_{n+1}}$. There is also a natural morphism   $\di\colon F\to C^0F$. One easily checks that $\di^2=0$.

\begin{thm} $C^\punto F$ is a finite and flasque resolution of  $F$.
\end{thm}

\begin{proof} By definition, $C^nF=0$ for $n>\dim X$. It is also clear that  $C^nF$ are flasque. Let us see that
\[ 0\to F\to C^0F \to \cdots\to C^{\dim X}F\to 0\] is an exact sequence.  We have to prove that  $(C^\punto F)(U_p)$ is a resolution of $F(U_p)$. One has a decomposition
\[ (C^nF)(U_p)= \proda{p=x_0<\cdots <x_n } F_{x_n}\times \proda{p<x_0<\cdots <x_n } F_{x_n} = (C^{n-1}F)(U^*_p)\times (C^nF)(U^*_p)\] with $U_p^*:=U_p-\{ p\}$; via this decomposition, the differential  $\di \colon (C^nF)(U_p) \to (C^{n+1}F)(U_p)$ becomes:
\[ \di(a,b)=(b-\di^*a,\di^*b)\] with $\di^*$ the differential of   $(C^\punto F)(U_p^*)$. It is immediate now that every cycle is a boundary.
\end{proof}

This theorem, together with De Rham's theorem (\cite{Godement}, Thm. 4.7.1), yields that the cohomology groups of a sheaf can be computed with the standard resolution, i.e., $H^i_\phi(U,F)=H^i\Gamma_\phi(U,C^\punto F)$, for any open subset $U$ of $X$, any family of supports $\phi$ and any sheaf $F$ of abelian groups on $X$.

\begin{cor} For any finite topological space $X$, any sheaf $F$ of abelian groups on $X$ and any family of supports $\phi$, one has
\[ H^n_\phi(X,F)=0,\quad \text{for any } n>\text{\rm dim} X.\] Moreover, if $F_p$ is a finitely generated $\ZZ$-module for any $p\in X$, then
$H^i_\phi(X,F)$ is a finitely generated $\ZZ$-module for any $i\geq 0$.
\end{cor}

Let $\M$ be an $\OO$-module, $U$ an open subset. For each $x\in U$ there is a natural map $\OO_x\otimes_{\OO(U)}\M(U)\to\M_x$. This induces a morphism $(C^n\OO)(U)\otimes_{\OO(U)}\M(U)\to (C^n\M)(U)$ and then a morphism of complexes of sheaves
$(C^\punto\OO)\otimes_\OO\M\to \C^\punto\M$.

\begin{prop}\label{qc-resolution} If $\M$ is quasi-coherent, then $(C^\punto\OO)\otimes_\OO\M\to \C^\punto\M$ is an isomorphism. Moreover, for any $p\in X$ and any open subset $U\subseteq U_p$, one has that
\[ \Gamma(U,C^\punto\OO)\otimes_{\OO_p}\M_p\to \Gamma(U,C^\punto \M)\] is an isomorphism.
\end{prop}

\begin{proof} Since $\M$ is quasi-coherent, for any $x\in U$, the natural map $\OO_x\otimes_{\OO_p}\M_p\to\M_x$ is an isomorphism. Hence, $(C^n\OO)(U)\otimes_{\OO_p}\M_p\to (C^n\M)(U)$ is an isomorphism, so we obtain the second part of the statement. The first part follows from the second, taking $U=U_p$.
\end{proof}

\medskip
\noindent{\it Integral functors.}
\medskip

For any ringed space $(X,\OO_X)$ we shall denote by $D(X)$ the (unbounded) derived category of complexes of $\OO_X$-modules and by $D_{\rm qc}(X)$ the faithful subcategory of complexes with quasi-coherent cohomology. We shall denote by $D(\text{{\rm Qcoh}(X)})$ the derived category of complexes of quasi-coherent $\OO_X$-modules. For a ring $A$, $D(A)$ denotes the derived category of complexes of $A$-modules.

Let $X,X'$ be two ringed finite spaces, and let $\pi\colon X\times X'\to X, \pi'\colon X\times X'\to X'$ be the natural projections. Given an object $\K\in D(X\times X')$, one defines the integral functor of kernel $\K$ by:

\[ \aligned \Phi_\K\colon D(X)&\to D (X')\\ \M &\mapsto \Phi_\K(\M)= \RR \pi'_*( \K\overset\LL\otimes \LL \pi^*\M)\endaligned\]

In general, if we take $\K$ in $D_{\rm qc}(X\times X')$, $\Phi_K$ does not map $D_{\rm qc}(X)$ into $D_{\rm qc}(X')$; the problem is that $\RR \pi'_*$ does not preserve quasi-coherence in general. However, we shall see that, for finite spaces, this holds.

\subsection{Basic cohomological properties of finite spaces}

For this subsection $X$ is a finite space, i.e., a ringed finite space with flat restrictions.

\begin{ejem}\label{coh-schemes} Let $(S,\OO_S)$ be a quasi-compact and quasi-separated scheme and $(X,\OO)$ the finite space associated to a (locally affine) finite affine covering. We already know (Theorem \ref{schemes}) that the morphism $\pi\colon S\to X$ yields an equivalence between the categories of quasi-coherent modules on $S$ and $X$. Moreover, if $\M$ is a quasi-coherent module on $S$, then
\[ H^i(S,\M)=H^i(X,\pi_*\M),\] since, for any $x\in X$, $(R^i\pi_*\M)_x=H^i(\pi^{-1}(U_x),\M)=0$ for $i>0$, because $\pi^{-1}(U_x)$ is an affine scheme. The topological analog is:

Let $S$ be a path connected, locally path connected and locally homotopically trivial topological space and  let $\U=\{ U_1,\dots,U_n\}$ be a (locally homotopically trivial) finite covering of $S$. Let $X$ be the associated finite topological space and $\pi\colon S\to X$ the natural continous map. We already know (Theorem \ref{fin-sp-assoc-top}) that the morphism $\pi\colon S\to X$ yields an equivalence between the categories of locally constant sheaves on $S$ and $X$. Moreover, if $F$ is a locally constant sheaf on $S$, then
\[ H^i(S,F)=H^i(X,\pi_*F),\] since, for any $x\in X$, $(R^i\pi_*F)_x=H^i(\pi^{-1}(U_x),F)=0$ for $i>0$, because $\pi^{-1}(U_x)$ is homotopically trivial.

\end{ejem}

\begin{thm}\label{qc-of-proj} Let $\pi\colon X\times X'\to X$ be the natural projection, with $X$ a  finite space. For any quasi-coherent sheaf $\M$ on $X\times X'$ and any $i\geq 0$,   $R^i\pi_*\M$ is quasi-coherent.
\end{thm}

\begin{proof} Let $p\in X$ and  $\pi'\colon U_p\times X'\to X'$ the natural projection. One has that $(R^i\pi_*\M)_p = H^i(U_p\times X',\M)= H^i(X', \pi'_*(\M_{\vert U_p\times X'}))$. By Theorem \ref{qc}, we have to prove that $$H^i(X',\pi'_*(\M_{\vert U_p\times X'}))\otimes_{\OO_p}\OO_{q}\to H^i(X',\pi''_*(\M_{\vert U_{q}\times X'}))$$ is an isomorphism for any  $p\leq q$, where $\pi''\colon U_{q}\times X'\to X'$ is the natural projection.

Let us denote   $\Nc = \pi'_*(\M_{\vert U_p\times X'}) $ and $\Nc' =\pi''_*(\M_{\vert U_{q}\times X'})$. Since
 $\OO_p\to\OO_{q}$ is flat, it is enough to prove that  $\Gamma (X', C^n \Nc)\otimes_{\OO_p}\OO_q \to \Gamma (X', C^n \Nc')$ is an  isomorphism. For any $x'\in X'$ one has
\[ \Nc_{x'} = \M_{(p,x')}\quad \text{ and }\quad \Nc'_{x'} = \M_{(q,x')}.\] Since $\M$ is quasi-coherent, $\Nc_{x'}\otimes_{\OO_p}\OO_q = \M_{(p,x')}\otimes_{\OO_{(p,x')}}\OO_{(q,x')} = \M_{(q,x')} =\Nc'_{x'}$. From the definition of $C^n$ it follows that $\Gamma (X', C^n \Nc)\otimes_{\OO_p}\OO_q = \Gamma (X', C^n \Nc')$; indeed,
\[ \Gamma (X', C^n \Nc)\otimes_{\OO_p}\OO_q = \proda{x'_0<\dots <x'_n} \Nc_{x'_n}\otimes_{\OO_p}\OO_q = \proda{x'_0<\dots <x'_n} \Nc'_{x'_n} = \Gamma (X', C^n \Nc')\]
\end{proof}

\begin{thm}\label{aciclico} Let  $X$ be a finite space,  $p\in X$ and $U\subset U_p$ an open subset. If $U $ is acyclic  (i.e., $H^i(U\ ,\OO)=0$ for any $i>0$), then
\begin{enumerate}
\item $\OO_p\to \OO(U )$ is flat.
\item For any quasi-coherent module $\M$  on $U_p$,
\[ H^i(U ,\M)=0,\qquad\text{ for } i>0\] and the natural morphism
\[ \M_p\otimes_{\OO_p}\OO(U )\to \M(U )\] is an isomorphism.

\end{enumerate}
\end{thm}

\begin{proof} By hypothesis, $(C^\punto \OO)(U )$ is a finite resolution of  $\OO(U )$. Moreover, $(C^i\OO)(U)$ is a flat $\OO_p$-module because $X$ is a finite space.  Hence $\OO(U )$ is a flat $\OO_p$-module. For the second part, one has an exact sequence of flat  $\OO_p$-modules
\[ 0\to \OO(U )\to (C^0\OO)(U )\to\cdots\to (C^n\OO)(U )\to 0\] Hence, the sequence remains exact after tensoring by  $\otimes_{\OO_p}\M_p$. One concludes by Proposition \ref{qc-resolution}.
\end{proof}


\begin{thm}\label{graphic} Let $f\colon X\to Y$ be a morphism, $\Gamma\colon X\to X\times Y$ its graphic and $\pi\colon X\times Y\to X$ the natural projection. For any quasi-coherent module $\M$ on $X$ the natural morphism
\[ \LL\pi^*\M\overset \LL\otimes \RR\Gamma_*\OO_X \to \RR\Gamma_*\M\] is an isomorphism (in the derived category).
\end{thm}

\begin{proof} For any $(x,y)$ in $X\times Y$, let us denote $U_{xy}=U_x\cap f^{-1}(U_y)$. The natural morphism
\[\Gamma(U_{xy},C^\punto\OO_X)\otimes_{\OO_x}\M_x\to \Gamma(U_{xy},C^\punto \M)\] is an isomorphism by Proposition \ref{qc-resolution}.
Now,  $$[\LL\pi^*\M\overset \LL\otimes \RR\Gamma_*\OO_X]_{(x,y)}= M_x\otimes_{\OO_x}\Gamma(U_{xy},C^\punto\OO_X)$$ because $\Gamma(U_{xy},C^\punto\OO_X)$ is a complex of flat $\OO_x$-modules; on the other hand,  $[\RR\Gamma_*\M]_{(x,y)}= \Gamma(U_{xy},C^\punto \M)$. We are done.
\end{proof}

\section{Affine finite spaces}

Let $(X,\OO)$ be an arbitrary ringed space,  $A=\Gamma(X,\OO)$ and $\pi\colon X \to (*,A)$ the natural morphism. Let $\M$ be an $\OO$-module.

\begin{defn} We say that $\M$ is  {\it acyclic} if $H^i(X,\M)=0$ for any $i>0$. We say that $X$ is  {\it acyclic} if $\OO$ is acyclic.
We say that $\M$ is  {\it generated by its global sections} if the natural map $\pi^*\pi_*\M\to \M$ is surjective. In other words, for any $x\in X$, the natural map
\[ M\otimes_A\OO_x\to \M_x,\qquad M=\Gamma(X,\M),\] is surjective.
\end{defn}

If $\M \to \Nc$ is surjective and $\M$ is generated by its global sections, then $\Nc$ too. If $f\colon X\to Y$ is a morphism of ringed spaces and $\M$ is an $\OO_Y$-module generated by its global sections, then $f^*\M$ is generated by its global sections.

\begin{defn} We say that $(X,\OO)$ is an {\it affine ringed space} if it is acyclic and $\pi^*$ (or $\pi_*=\Gamma(X,\quad )$) gives an equivalence between the category of $A$-modules and the category of quasi-coherent $\OO$-modules. We say that $(X,\OO)$ is {\it quasi-affine} if every quasi-coherent $\OO$-module is generated by its global sections.  We say that $(X,\OO)$ is {\it Serre-affine} if every quasi-coherent module is acyclic.
\end{defn}

Obviously, any affine ringed space is quasi-affine. If $X$ is a finite ringed space with a minimum, then it is affine, quasi-affine and Serre-affine (Corollary \ref{corqc} and Proposition \ref{aciclicity}). Hence any finite ringed space is locally affine.

Before we see the basic properties and relations between these concepts on a finite space, let us see some examples for (non-finite) ringed spaces.

\begin{ejems}
\begin{enumerate}
\item (See \cite{NavarroSancho}) Let $(X,\C^\infty_X)$ be a  Haussdorff differentiable manifold (more generally, a differentiable space) with a countable basis. Any $\C^\infty_X$-module is quasi-coherent, acyclic and generated by its global sections. Hence $X$ is quasi-affine and Serre-affine. If we denote $A=\C^\infty(X)$, the functor
    \[ \aligned \{ \C^\infty_X-\text{modules}\}& \to \{ A-\text{modules}\} \\ \M &\mapsto \M(X)\endaligned\] is fully faithful and it has a left inverse $M\mapsto \pi^* M$. If $X$ is {\it compact}, then it is an equivalence, i.e., $X$ is affine  (this result is not in \cite{NavarroSancho}, but it is not difficult to prove). If $X$ is not compact (for example $X=\RR^n$), then it is not an equivalence. So $(\RR^n,\C^\infty_{\RR^n})$ is an example of an acyclic, quasi-affine and Serre-affine ringed space which is not affine.
\item Let $X=\Spec A$ be an affine scheme ($\OO=\widetilde A$ the sheaf of localizations). Then it is affine, quasi-affine and Serre-affine. A scheme $X$ is affine (in the usual sense of schemes) if and only if it is affine (in our sense) or Serre-affine (Serre's criterion for affineness). A quasi-compact scheme $X$ is quasi-affine if and only if it is an open subset of an affine scheme.
\item Let $(S,\OO_S)$ be a quasi-compact quasi-separated scheme and $\pi \colon S\to X$ the finite space associated to a (locally affine) finite  covering of $S$. Then $S$ is an affine scheme if and only if $X$ is an affine finite space. Even more, an open subset $U$ of $X$ is affine if and only if $\pi^{-1}(U)$ is affine.
\end{enumerate}
\end{ejems}

From now on, assume that $(X,\OO)$ is a finite space.

\begin{thm}\label{AffineFinSp} Let $X$ be a finite space. The following conditions are equivalent:
\begin{enumerate}
\item $X$ is affine.
\item $X$ is acyclic and quasi-affine.
\item $X$ is quasi-affine and Serre-affine.

\end{enumerate}
\end{thm}

\begin{proof} (1) $\Rightarrow$ (2) is immediate. (2) $\Rightarrow$ (3). We have to prove that any quasi-coherent module $\M$ is acyclic. By hypothesis, $\pi^*M\to\M$ is surjective, with  $M=\M(X)$. Since $M$ is a quotient of a free  $A$-module, $\M$ is a quotient of a free $\OO$-module $\Lc$.
We have an exact sequence $0\to \K\to \Lc\to \M\to 0$. Since $X$ is acyclic, $H^i(X,\Lc)=0$ for any $i>0$. Then $H^d(X,\M)=0$, for $d=\dim X$. That is, we have proved that $H^d(X,\M)=0$ for {\it any} quasi-coherent module $\M$. Then $H^d(X,\K)=0$, so $H^{d-1}(X,\M)=0$, for {\it any} quasi-coherent module $\M$;  proceeding in this way we obtain that  $H^i(X,\M)=0$ for any $i>0$ and any quasi-coherent $\M$.

(3) $\Rightarrow$ (1). By hypothesis $X$ is acyclic and  $\pi_*$ is an exact functor over the category of quasi-coherent  $\OO$-modules. Let us see that $M\to \pi_*\pi^*M$ is an isomorphism for any  $A$-module $M$. If $M=A$, there is nothing to say. If $M$ is a free module, it is immediate. Since any $M$ is a cokernel of free modules, one concludes (recall the exactness of $\pi_*$). Finally, let us see that $\pi^*\pi_*\M\to \M$ is an isomorphism for any quasi-coherent $\M$. The surjectivity holds by hypothesis. If $K$ is the kernel, taking $\pi_*$ in the exact sequence
\[ 0\to K\to \pi^*\pi_*\M\to\M\to 0\] and taking into account that  $\pi_*\circ \pi^*=\Id$, we obtain that $\pi_*K=0$. Since $K$ is generated by its global sections, it must be  $K=0$.
\end{proof}

Let us see that in the topological case, being affine is equivalent to being homotopically trivial:

\begin{thm} {\rm (\cite{Quillen}, Proposition 7.8)} Let $X$ be a  connected finite topological space $(\OO=\ZZ$). Then $X$ is affine if and only if $X$ is homotopically trivial.
\end{thm}

\begin{proof} Quasi-coherent modules are locally constant sheaves, by Theorem \ref{qc-fts}.  Now, for any abelian group $G$,  sheaves locally isomorphic to the constant sheaf $G$ are in bijection with representations of $\pi_1(X)$ in $G$. Hence a connected topological space $X$ is affine if and only if it is acyclic and any locally constant sheaf is constant, i.e., any representation of $\pi_1(X)$ on any abelian group is trivial. This amounts to say that $X$ is acyclic and simply connected. By the duality between homology and cohomology groups, this is equivalent to say that $X$ is simply connected and all its homology groups are zero. One concludes by Hurewicz's Theorem.
\end{proof}

\begin{prop}\label{tens-affine} Let $X$ be an affine finite space, $A=\OO(X)$. For any quasi-coherent modules $\M,\M'$ on $X$, the natural morphism $\M(X)\otimes_A\M'(X) \to(\M\otimes_\OO \M')(X)$ is an isomorphism.
\end{prop}
\begin{proof} For any $A$-modules $M,N$ one has an isomorphism $\pi^*M\otimes_\OO \pi^*N\overset\sim \to \pi^*(M\otimes_AN)$. One concludes because $X$ is affine.
\end{proof}

\begin{prop}\label{derivedcat-affine} Let $X$ be an affine finite space, $A=\OO(X)$. Then
\begin{enumerate}
\item For any $p\in X$, the natural map $A\to \OO_p$ is flat. In other words, the functor $$\pi^*\colon \{ A-\text{modules}\}\to \{\OO-\text{modules}\}$$ is exact.
\item The natural (injective) morphism $A\to \underset{p\in X}\prod \OO_p$ is faithfully flat.
\item The natural functors
\[  \xymatrix{ D({\rm\text{\rm Qcoh}}(X))\ar[rr]\ar[rd]_{\Gamma(X,\underline\quad)} &  & D_{\rm\text{\rm qc}}(X)\ar[ld]^{\RR\Gamma(X,\underline\quad)} \\  &  D(A) & }\] are equivalences.
\end{enumerate}
\end{prop}

\begin{proof} (1) Lets us see that $\pi^*$ is exact. It suffices to see that it is left exact. Let $M\to N$ be a injective morphism of $A$-modules and let $\K$ be the kernel of $\pi^*M\to\pi^*N$. Since $X$ is affine, $\pi_*\pi^*=\Id$; hence $\pi_*\K=0$ and then $\K=0$ because $\K$ is quasi-coherent and $X$ is affine.

(2) Since $A\to\OO_p$ is flat, it remains to prove that $\Spec( \underset{p\in X}\prod \OO_p)\to \Spec A$ is surjective. Let $\pp$
be a prime ideal of $A$ and $k(\pp)$ its residue field. Since $X$ is affine, $\pi^*k(\pp)$ is a (non-zero) quasi-coherent module on $X$, hence there exists $p\in X$ such that $(\pi^*k(\pp))_{\vert U_p}$ is not zero. This means that $\OO_p\otimes_A k(\pp)$ is not zero, so the fiber of $\pp$ under the morphism $\Spec\OO_p\to \Spec A$ is not empty.

(3) $\pi_*\colon {\rm Qcoh}(X)\to \{ A-{\rm modules}\}$ is exact because $X$ is affine (hence Serre-affine), and $\pi^*\colon \{ A-{\rm modules}\}\to {\rm Qcoh}(X) $ is exact by (1). Since $X$ is affine, one concludes that $$\pi_*\colon D({\rm\text{\rm Qcoh}}(X))\to D(A)$$ is an equivalence (with inverse $\pi^*$). To conclude, it suffices to see that if $\M^\punto$ is a complex of $\OO$-modules with quasi-coherent cohomology, the natural morphism $ \pi^*\RR\pi_*\M^\punto \to \M^\punto$ is a quasi-isomorphism. Since $\HH^i(\M^\punto)$ are quasi-coherent and $X$ is affine, one has $H^j(X,\HH^i(\M^\punto))=0$ for any $j>0$. Then $H^j(X,\M^\punto)=H^0(X,\HH^j(\M^\punto))$; in other words, $H^j(\RR\pi_*\M^\punto)=\pi_*\HH^j(\M^\punto)$. Then
$$\HH^j(\pi^*\RR\pi_*\M^\punto)\overset{(1)}= \pi^* H^j(\RR\pi_*\M^\punto)=\pi^*\pi_*\HH^j(\M^\punto)$$ and $\pi^*\pi_*\HH^j(\M^\punto)\to \HH^j(\M^\punto)$ is an isomorphism because $\HH^j(\M^\punto)$ is quasi-coherent and $X$ is affine.
\end{proof}

\begin{cor} If $V\subseteq U$ is an inclusion between affine open subsets of a finite space $X$, then the restriction morphism $\OO(U)\to\OO(V)$ is flat.
\end{cor}

\begin{proof} By Theorem \ref{derivedcat-affine}, the composition $\OO(U)\to \OO(V)\to \underset{p\in V}\prod \OO_p$ is flat   and $\OO(V)\to \underset{p\in V}\prod \OO_p$ is faithfully flat. Conclusion follows.
\end{proof}

\begin{cor}\label{integral-functors} {\rm (1)} Let $f\colon X\to Y$ be a morphism between finite spaces. If $\M^\punto$ is a complex of $\OO_Y$-modules with quasi-coherent cohomology, then $\LL f^*\M^\punto$ is a complex of $\OO_X$-modules with quasi-coherent cohomology. Hence one has a functor
\[\LL f^*\colon D_{\rm\text{\rm qc}}(Y)\to D_{\rm\text{\rm qc}}(X)\]

{\rm (2)} Let $\M^\punto,\Nc^\punto$ be two complexes of $\OO$-modules on a finite space $X$. If $\M$ and $\Nc$ have quasi-coherent cohomology, so does $\M^\punto\overset\LL\otimes\Nc^\punto$. So one has a functor
\[\overset\LL\otimes \colon D_{\rm\text{\rm qc}}(X)\times D_{\rm\text{\rm qc}}(X)\to D_{\rm\text{\rm qc}}(X)\]

{\rm (3)} Let $X$ and $Y$ be two finite spaces and let $\K\in D_{\rm qc}(X\times Y)$. Then, the integral functor $\Phi_\K\colon D(X)\to D(Y)$ maps $D_{\rm qc}(X)$ into $D_{\rm qc}(Y)$.
\end{cor}

\begin{proof} (1) Since being quasi-coherent is a local question, we may assume that $Y$ is affine. Let us denote $\pi\colon Y\to (*,A)$ the natural morphism, with $A=\OO_Y(Y)$. By Proposition \ref{derivedcat-affine},  $\M^\punto \simeq \pi^* M^\punto$ for some complex of $A$-modules $M^\punto$, and then $\LL f^* \M^\punto \simeq \LL\pi_X^* M^\punto$, with $\pi_X=\pi\circ f$; it is clear that $\LL\pi_X^* M^\punto$ has quasi-coherent cohomology (in fact, it is a complex of quasi-coherent modules).

(2) Again, we may assume that $X$ is affine, and then $\M^\punto\simeq \pi^*M^\punto$, $\Nc^\punto\simeq \pi^*N^\punto$. Then
\[ \M^\punto\overset\LL\otimes\Nc^\punto \simeq  \pi^*M^\punto\overset\LL\otimes \pi^*N^\punto \simeq \pi^*( M^\punto\overset\LL\otimes N^\punto).\]

(3) It follows from (1), (2) and Theorem \ref{qc-of-proj}.
\end{proof}

\begin{rem} In \cite{Ladkani}, it is studied what properties of a finite topological space are invariant under {\it derived equivalence}, when one works with the bounded derived category of sheaves of $k$-vector spaces. For example, the number of points or the Betti numbers are invariant under this equivalence. However, if one works with the derived category of quasi-coherent $\ZZ$-modules (i.e., locally constant sheaves of abelian groups), the number of points is no longer an invariant, since two homotopic posets have the same quasi-coherent derived category but different number of points.  Our feeling is that the derived category of all sheaves (of $k$-vector spaces) is not the natural object to study, in the same way that, in the context of algebraic geometry, one does not usually consider the derived category of all sheaves of $k$-vector spaces on a $k$-scheme, but  the derived category of quasi-coherent or coherent sheaves. If one desires a derived equivalence which is compatible with homotopy (i.e., if two posets are homotopic then they are derived equivalent), then the derived category of locally constant sheaves would be a good candidate.  An interesting question comes here: let $X$ and $Y$ be two posets which are derived equivalent with respect the derived category of locally constant sheaves of abelian groups; what homotopical invariants  are derived invariants?

Even for a small poset, the category of all sheaves is very big; for example, take a (connected) poset (of three elements) with two closed points and one generic point. The category of sheaves of $k$-vector spaces on this poset contains the category of quasi-coherent modules on any irreducible algebraic curve over $k$ (or any separated $k$-scheme which can be covered by two affine open subschemes).
\end{rem}

\noindent{\it Relative notions.}
\medskip

To conclude this section, we give the relative notions of affineness, quasi-affineness, etc, and study the behaviour of these concepts under basic operations as direct products or base change. We do not go further because the good behaviour of these concepts will hold for schematic finite spaces and schematic morphisms, as we shall see in the next section.  An affine space can still have a lot of pathologies; however, affine schematic spaces and affine schematic morphisms share all the basic properties of affine schemes and affine morphisms of schemes. In  Algebraic Geometry, every affine scheme is semi-separated (indeed, it is separated). This does not happen for finite spaces; thus,  finite spaces which are most similar to affine schemes are those affine spaces which are semi-separated (equivalently, schematic). The same can be said of morphisms.

\begin{defn} We say that a morphism  $f\colon X\to Y$ preserves quasi-coherence if for any quasi-coherent module  $\M$ on $X$, $f_*\M$ is quasi-coherent. We say that $R^if$ preserves quasi-coherence if for any quasi-coherent module  $\M$ on $X$, $R^if_*\M$ is quasi-coherent. We say that $\RR f$ preserves quasi-coherence if $R^if_*$ preserves quasi-coherence for any $i$.
\end{defn}

If $Y$ is a punctual space, any morphism $f\colon X\to Y$ preserves quasi-coherence (and $\RR f$ preserves quasi-coherence).

\begin{defn} Let $f\colon X\to Y$ be a morphism,  $\M$ an $\OO_X$-module.
\begin{enumerate}
\item We say that  $\M$ is generated by its global sections over $Y$ if the natural map  $f^*f_*\M\to\M$ is surjective.
\item We say that $\M$ is $f$-acyclic if $R^if_*\M=0$ for any $i>0$. We say that $X$ is $f$-acyclic if $\OO_X$ is $f$-acyclic. In this case we also say that $f$ is acyclic.
\item We say that $f$ is quasi-affine if it preserves quasi-coherence and every quasi-coherent  $\OO_X$-module is generated by its global sections over $Y$.
\item We say that $f$ is Serre-affine if it preserves quasi-coherence and  every quasi-coherent  $\OO_X$-module is $f$-acyclic.
\item We say that   $f$ is affine  if it preserves quasi-coherence and   $f^{-1}(U_y)$ is affine for any $y\in Y$.
\end{enumerate}
\end{defn}

\begin{rem}\label{relative=absolute} If $Y$ is a punctual space, the relative notions coincide with the absolute ones, i.e.: Let $f\colon X\to Y$ be a morphism, $Y$ a punctual space and $\M$ an $\OO_X$-module. Then
\begin{enumerate}
\item   $\M$ is generated by its global sections over  $Y$ if and only if  $\M$ is generated by its global sections.
\item   $\M$ is $f$-acyclic if and only if it is acyclic.
\item   $f$ is quasi-affine if and only if $X$ is quasi-affine.
\item   $f$ is  Serre-affine if and only if $X$ is Serre-affine.
\item   $X$ is $f$-acyclic if and only if it is acyclic.
\item $f$ is affine if and only if $X$ is affine.
\end{enumerate}
\end{rem}

Preserving quasi-coherence, being Serre-affine and being quasi-affine are stable under composition.

\begin{ejem}\label{ejem-affine} Let $(X,\OO)$ be a finite space and $\Bc$ a quasi-coherent $\OO$-algebra; that is, $\Bc$ is a sheaf of rings on $X$ endowed with a morphism of sheaves of rings $\OO\to\Bc$, such that $\Bc$ is quasi-coherent as an $\OO$-module. Then $(X,\Bc)$ is a finite space (it has flat restrictions). Moreover, the identity on $X$ and the morphism $\OO\to \Bc$ give a morphism of ringed finite spaces $(X,\Bc)\to (X,\OO)$. This morphism is affine.
\end{ejem}

\begin{thm}\label{projection-formula} Let $f\colon X\to Y$ be an affine morphism.

(1) For any quasi-coherent module $\M$ on $X$ and any quasi-coherent module $\Nc$ on $Y$, the natural morphism $\Nc\otimes f_*\M\to f_*(f^*\Nc \otimes \M)$ is an isomorphism.

(2) A morphism $\M\to\M'$ between quasi-coherent modules on $X$ is an isomorphism if and only if the induced morphism $f_*\M\to f_*\M'$ is an isomorphism.

(3) If $f_*\OO_X=\OO_Y$, then $f_*$ and $f^*$ yield an equivalence between the categories of quasi-coherent modules on $X$ an $Y$.

\end{thm}

\begin{proof} (1) Let $y\in Y$ and let us denote $N=\Nc_y$, $M=\M(f^{-1}(U_y))$, $A=\OO_y$, $B=\OO_X(f^{-1}(U_y))$. Taking the stalk at $y$, we obtain the morphism $N\otimes_A M \to \Gamma(f^{-1}(U_y), f^*\Nc\otimes\M)$. By Proposition \ref{tens-affine}, $\Gamma(f^{-1}(U_y), f^*\Nc\otimes\M)= \Gamma(f^{-1}(U_y), f^*\Nc)\otimes_B  M$. Finally, $\Gamma(f^{-1}(U_y), f^*\Nc)=N\otimes_AB$, because $f^{-1}(U_y)$ is affine. Conclusion follows.

(2) For each $y\in Y$, let us denote $X_y=f^{-1}(U_y)$. Now, $\M\to \M'$ is an isomorphism if and only if $\M_{\vert X_y}\to \M'_{\vert X_y}$ is an isomorphism for any $y$. Since $X_y$ is affine, this is an isomorphism if and only if it is an isomorphism after taking global sections, i.e., iff $(f_*\M)_y\to (f_*\M')_y$ is an isomorphism.

(3) For any quasi-coherent module $\Nc$ on $Y$, the natural morphism $\Nc\to f_*f^*\Nc$ is an isomorphism by (1) and the hypothesis $f_*\OO_X=\OO_Y$. For any quasi-coherent module $\M$ on $X$, the natural morphism $f^*f_*\M\to \M$ is an isomorphism by (2), since $f_*(f^*f_*\M)=f_*\M\otimes f_*\OO_X$ and $f_*\OO_X=\OO_Y$.
\end{proof}

\begin{prop}\label{Serre-Quasi-Afin} If $f\colon X\to Y$ is an affine morphism, then it is  Serre-affine and quasi-affine.
\end{prop}

\begin{proof} Since $f$ is affine, it preserves quasi-coherence. Let us see that it is Serre-affine.  Let $\M$ be a quasi-coherent module on  $X$. We have to prove that  $R^if_*\M=0$ for $i>0$. But $(R^if_*\M)_y=H^i(f^{-1}(U_y),\M)=0$ for $i>0$ because $f^{-1}(U_y)$ is affine, hence Serre-affine, by Theorem \ref{AffineFinSp}.

Let us see now that $f$ is quasi-affine. We have to prove that   $f^*f_*\M\to\M$ is surjective. The question is local on   $Y$, so we may assume that   $Y=U_y$  and  $X$ is affine. Since $f_*\M$ is quasi-coherent and $X$ is affine, the morphism $f^*f_*\M\to\M$ is equivalent to the morphism $M\otimes_AB\to M$, with $M$, $A$, $B$ the global sections of  $\M$, $\OO_Y$, $\OO_X$, which is  obviously surjective.
\end{proof}

For the rest of this subsection, $f\colon X\to Y$ is a morphism between finite spaces, $S$ is another finite space and $f\times 1\colon X\times S\to Y\times S$ is the induced morphism. The following technical results will be used in sections 5. and 6.

\begin{prop}\label{preservacion-cuasi}  If $R^if$ preserves quasi-coherence, so does $ R^i(f\times 1)$. Consequently, given morphisms $f\colon X\to Y$ and $f'\colon X'\to Y'$, if $\RR f$ and $\RR f'$ preserve quasi-coherence, then $\RR (f\times f')$ preserves quasi-coherence, with $f\times f'\colon X\times X'\to Y\times Y'$.
\end{prop}
\begin{proof} Let $\M$ be a quasi-coherent module on $X\times S$. Let us see that $R^i(f\times 1)_*\M$ is quasi-coherent. Since the question is local, we may assume that  $S=U_s$. We have to prove that the natural morphisms
\[\aligned \left[ R^i(f\times 1)_*\M\right]_{(y,s)}\otimes_{\OO_y}\OO_{y'}&\to [R^i(f\times 1)_*\M]_{(y',s)}\\ [R^i(f\times 1)_*\M]_{(y,s)}\otimes_{\OO_s}\OO_{s'}&\to [R^i(f\times 1)_*\M]_{(y,s')}\endaligned\] are isomorphisms for any $y\leq y'$ in $Y$ and $s'\in U_s$. For the first, we have
\[ [R^i(f\times 1)_*\M]_{(y,s)}=H^i (f^{-1}(U_y)\times U_s,\M) = [R^if_* (\pi_*\M)]_y\] with $\pi\colon X\times U_s\to X$ the natural projection. Since $R^if_* (\pi_*\M)$ is quasi-coherent, one concludes the first  isomorphism. The second follows from the fact that for every open subset $U$ of $X$, the natural morphism  $H^i(U\times U_s,\M)\otimes_{\OO_s}\OO_{s'}\to H^i(U\times U_{s'},\M)$ is an isomorphism, because $R^i\pi'_*(\M_{\vert U\times U_s})$ is quasi-coherent, with $\pi'\colon U\times U_s\to U_s$ the natural projection.

For the consequence, put $f\times f'$ as the composition of $f\times 1$ and $1\times f'$.
\end{proof}

\begin{prop}\label{fx1}  If $f$ is Serre-affine (resp. quasi-affine), so is $f\times 1$.
\end{prop}

\begin{proof} Assume $f$ is Serre-affine. Then $f\times 1$ preserves quasi-coherence by Proposition \ref{preservacion-cuasi}.
 Let $\M$ be a quasi-coherent module on $X\times S$. We have to prove that  $R^i(f\times 1)_*\M=0$ for any $i>0$. For each $(y,s)\in Y\times S$, one has:
\[ [R^i(f\times 1)_*\M]_{(y,s)}=H^i(f^{-1}(U_y)\times U_s,\M)=H^i(f^{-1}(U_y), \pi_*(\M_{\vert X\times U_s}))\] with $\pi\colon X\times U_s\to X$, because $R^i\pi_*=0$. Now, $$H^i(f^{-1}(U_y), \pi_*(\M_{\vert X\times U_s}))= [R^i f_*(\pi_*(\M_{\vert X\times U_s}))]_y=0$$ because $\pi_*(\M_{\vert X\times U_s})$ is quasi-coherent and $f$ is Serre-affine.

Assume now that $f$ is quasi-affine We have to prove that
$ (f\times 1)^*(f\times 1)_*\M\to \M$ is surjective, i.e., for each point $(x,s)\in X\times S$ the map (denote  $y=f(x)$)
\[ [(f\times 1)_*\M]_{(y,s)}\otimes_{\OO_{(y,s)}}\OO_{(x,s)}\to\M_{(x,s)}\] is surjective.  This is equivalent to prove that  the map
\[ \Gamma(f^{-1}(U_y)\times U_s,\M)\otimes_{\OO_y}\OO_x\to \Gamma(U_x\times U_s,\M)\] is surjective. If we denote  $\pi\colon X\times U_s\to X$ and $\Nc=\pi_*(\M_{\vert X\times U_s})$, the latter morphism is
\[ (f_*\Nc)_y\otimes_{\OO_y}\OO_x\to \Nc_x,\] which is surjective because  $\Nc$ is quasi-coherent and $f$ is quasi-affine.
\end{proof}

\begin{cor}\label{product-affine} If $X$ and $Y$ are two   affine (resp. quasi-affine, Serre-affine) spaces, then   $X\times Y$ is affine (resp. quasi-affine, Serre-affine). Moreover, if $A=\OO_X(X)$ and $B=\OO_Y(Y)$ are flat $k$-algebras, and $X$, $Y$ are affine, then $\Gamma(X\times_kY,\OO_{X\times_kY})=A\otimes_kB$.
\end{cor}

\begin{proof} Assume that $X$ and $Y$ are Serre-affine. The projection $\pi\colon X\times Y\to X$ is Serre-affine by the preceding theorem and the Serre-affineness of  $Y$. Then, the composition $X\times Y\to X\to (*,A)$ is Serre-affine, i.e., $X\times Y$ is Serre-affine (Remark \ref{relative=absolute}). The same argument proves that $X\times Y$ is quasi-affine if $X$ and $Y$ are so.

If $X$ and $Y$ are affine, then they are Serre-affine and quasi-affine. Hence $X\times Y$ is Serre-affine and quasi-affine, so it is  affine. Assume also that $A=\OO_X(X)$ and $B=\OO_Y(Y)$ are flat $k$-algebras, and let us prove that $\Gamma(X\times_kY,\OO_{X\times_kY})=A\otimes_kB$. Let $\pi\colon X\times_k Y\to X$ be the natural projection. If suffices to see that the natural map $\OO_X\otimes_k B\to {\pi}_*\OO_{X\times_kY}$ is an isomorphism, where $\OO_X\otimes_k B$ is the sheaf on $X$ defined by $(\OO_X\otimes_k B)(U)=\OO_X(U)\otimes_k B$ (which is a sheaf because $k\to\OO_X(U)$ is flat). The question is local on $X$, hence we may assume that $X=U_x$ (notice that $\OO_x$ is a flat $k$-algebra, by Proposition \ref{derivedcat-affine}), and we have to prove that $\Gamma(U_x\times_kY,\OO_{U_x\times_kY})=\OO_x\otimes_k B$. If $\pi'\colon U_x\times_k Y\to Y$ is the natural projection, it suffices to see that the natural morphism $\OO_x\otimes_k \OO_Y\to \pi'_*\OO_{U_x\times_kY}$ is an isomorphism, which is immediate by taking the stalk at any $y\in Y$.
\end{proof}

\begin{cor} If $f\colon X\to Y$ is affine, then $f\times 1\colon X\times S\to Y\times S$ is affine,  for any finite space  $S$.
\end{cor}

\begin{proof} $f\times 1$ preserves quasi-coherence by Proposition \ref{preservacion-cuasi}. Moreover, $(f\times 1)^{-1}(U_y\times U_s)= f^{-1}(U_y)\times U_s$ is a product of affine spaces, hence it is affine.
\end{proof}

\section{Schematic finite spaces}

Let $X$ be a finite space, $\delta\colon X\to X\times_k X$ the diagonal morphism (we shall  see now that $k$ is irrelevant). 

\begin{defn} We say that a finite space $X$ is  {\it schematic} if  $R^i\delta_*\OO$ is quasi-coherent for any $i$ (for the sake of brevity, we shall say that $\RR\delta_*\OO$ is quasi-coherent).

For any $p,q\in X$, let us denote $U_{pq}=U_p\cap U_q$, and $\OO_{pq}=\OO (U_{pq})$. Notice that $\OO_{pq}=\OO_{qp}$ and

\begin{equation} \label{delta-qc} (\delta_*\OO)_{(p,q)}=\OO_{pq}\end{equation}

One has natural morphisms $\OO_p\to \OO_{pq}$ and $\OO_q\to \OO_{pq}$. From the equality $(R^i\delta_*\OO)_{(p,q)}=H^i(U_{pq},\OO)$, it follows that being  schematic   means that for any $p,q\in X$,  any $p'\geq p$  and any $i\geq 0$, the natural morphism
\[ H^i(U_{pq},\OO)\otimes_{\OO_p}\OO_{p'}\to H^i(U_{p'q},\OO) \] is an isomorphism. In particular, $k$ is irrelevant.

\end{defn}

\begin{ejems}
\begin{enumerate} \item If $X$ is the finite space associated to a quasi-compact and quasi-separated scheme $S$ and a locally affine covering $\U$, then $X$ is schematic. It is a consequence of the following fact: if $U$ is an affine scheme and $V\subset U$ is a quasi-compact open subset, then for any affine open subset $U'\subset U$, the natural morphism
\[ H^i(V,\OO_S )\otimes_{\OO_S (U)}\OO_S (U')\to H^i(V\cap U',\OO_S)\] is an isomorphism.


\item A finite topological space $X$ (i.e. $\OO=\ZZ$) is schematic if and only if each connected component is irreducible.
\end{enumerate}
\end{ejems}

We shall first study the implications of a weaker  condition,  the quasi-coherence of $\delta$.

\subsection{Quasi-coherent finite spaces}

\begin{defn} A morphism $f\colon X\to Y$ is called {\it quasi-coherent} if $f_*\OO_X$ is quasi-coherent. A finite space $X$ is called {\it quasi-coherent} if $\delta\colon X\to X\times X$ is quasi-coherent.
\end {defn}

A morphism $f\colon X\to Y$ is quasi-coherent if and only if for any open subset $V$ of $Y$ the morphism $f^{-1}(V)\to V$ is quasi-coherent.

\begin{prop} A module $\M$ on $X\times X$ is quasi-coherent if and only if for any $p\leq p'$, $q\leq q'$, the natural morphisms
\[ \M_{(p,q)}\otimes_{\OO_p} {\OO_{p'}}\to \M_{(p',q)} \quad \text{and}\quad \M_{(p,q)}\otimes_{\OO_q} {\OO_{q'}}\to \M_{(p,q')}\] are isomorphisms.
\end{prop}

\begin{proof} For any module $\M$ on $X\times X$ and any $(p,q)\leq (p',q')$ one has:
\[ \M_{(p,q)}\otimes_{\OO_{(p,q)}}\OO_{(p',q')} =  \M_{(p,q)}\otimes_{\OO_{(p,q)}}\OO_{(p',q)}\otimes_{\OO_{(p',q)}}\OO_{(p',q')}\] and $\M_{(p,q)}\otimes_{\OO_{(p,q)}}\OO_{(p',q)}= \M_{(p,q)}\otimes_{\OO_p} {\OO_{p'}}$, $\M_{(p,q)}\otimes_{\OO_{(p,q)}}\OO_{(p,q')}= \M_{(p,q)}\otimes_{\OO_q} {\OO_{q'}}$, because $\OO_{(p,q)}=\OO_p\otimes\OO_q$. One concludes by  Theorem \ref{qc}.
\end{proof}

From this Proposition and  formula (\ref{delta-qc}), one obtains:

\begin{prop}\label{fqc1} $X$ is  quasi-coherent if and only if for any $p,q\in X$ and any $p'\geq p$ the natural morphism
\[ \OO_{pq}\otimes_{\OO_p}\OO_{p'}\to \OO_{p'q}\] is an isomorphism.
\end{prop}

\begin{prop}\label{finite-qc-char1} The following conditions are equivalent (all maps considered are the natural inclusions):
\begin{enumerate}
\item $X$ is quasi-coherent.
\item For any $q\in X$, $U_q\hookrightarrow X$ is quasi-coherent.
\item For any open subset $U$ of $X$, $U \hookrightarrow X$ is quasi-coherent.
\item For any open subsets $U\subseteq V$ of $ X$, $U \hookrightarrow V$ is quasi-coherent.
\item For any  $p\leq q$, $U_q \hookrightarrow U_p$ is quasi-coherent.
\end{enumerate}
\end{prop}

\begin{proof} 1. $\Leftrightarrow$ 2. Let us denote $i\colon U_q \hookrightarrow X$ the inclusion. For any $p\in X$, one has $(i_*\OO_{\vert U_q})_p=\OO_{pq}$. One concludes by Proposition \ref{fqc1}.

2. $\Rightarrow$ 3. For any open subset $V\subseteq X$, let us denote $\OO_V=i_*\OO_{\vert V}$, with $i$ the inclusion of $V$ in $X$.
Let $U_1,\dots, U_n$ be a covering of $U$ by minimal open subsets and $\{U_{ijk}\}_k$ a covering of $U_i\cap U_j$ by minimal open subsets. One has an exact sequence
\[ 0\to \OO_U\to \oplus \OO_{U_i}\to \oplus \OO_{U_{ijk}}\] so $\OO_U$ is the kernel of a morphism between quasi-coherent modules, hence it is quasi-coherent.

3. $\Rightarrow$ 4. $\Rightarrow$ 5. are obvious. Finally, let us see that 5. $\Rightarrow$ 2. If $U_q \hookrightarrow U_p$ is quasi-coherent for any $q\geq p$, then $U_p$ is quasi-coherent for any $p$ ($2.\Rightarrow 1.$). Then, for any $p,q\in X$, $U_p\cap U_q\hookrightarrow U_p$ is quasi-coherent (since $U_p$ satisfies 3.). This says that   $U_q\hookrightarrow X$ is quasi-coherent.
\end{proof}

\begin{rem} It follows from this theorem that being quasi-coherent is a local question:    $X$ is quasi-coherent if and only if every minimal open subset  $U_p$ is quasi-coherent.
\end{rem}

\begin{prop}\label{finite-qc-char2} The following conditions are equivalent:
\begin{enumerate}
\item $X$ is quasi-coherent.
\item For any open subset $U$ of $ X$ and any $p\leq p'$, the natural morphism \[\OO(U\cap U_p)\otimes_{\OO_p}\OO_{p'} \to \OO(U\cap U_{p'})\] is an isomorphism.
\item  For any $q_1\geq p \leq q_2$, the natural map $\OO_{q_1}\otimes_{\OO_p}\OO_{q_2}\to \OO_{q_1q_2}$ is an isomorphism.
\end{enumerate}
\end{prop}

\begin{proof} Condition 2. says that the inclusion   $U\hookrightarrow X$ is quasi-coherent. Condition 3. says that the inclusion $U_{q_1}\hookrightarrow U_p$ is quasi-coherent. One concludes by Proposition \ref{finite-qc-char1}.
\end{proof}

The following proposition will not be used in the rest of the paper but gives interesting information of quasi-coherent finite spaces.

\begin{prop} If $X$ is quasi-coherent, then
\begin{enumerate}
\item   For any  $ p \leq q$, the natural map $\OO_q\otimes_{\OO_p}\OO_q\to \OO_q, a\otimes b\mapsto ab$, is an isomorphism. In other words, $\Spec\OO_q\to\Spec \OO_q$ is a flat monomorphism of affine schemes.
\item The module of differentials $\Omega_X$ (defined by $(\Omega_X)_p=\Omega_{\OO_p/k})$ is quasi-coherent.
\end{enumerate}
\end{prop}

\begin{rem} By definition, $X$ is quasi-coherent if $\delta_*\OO$ is quasi-coherent. We could replace $\OO$ by any other quasi-coherent $\OO$-module $\M$, and study those $\M$ satisfying that $\delta_*\M$ is quasi-coherent. One could easily reproduce formula \ref{delta-qc} and Propositions \ref{finite-qc-char1} and \ref{finite-qc-char2}. We shall not do this  because we shall see that, if $X$ is schematic, then every quasi-coherent module has this property (i.e., $\delta$ preserves quasi-coherence).
\end{rem}

\begin{thm}\label{local-structure-affine} Let $X$ be a an affine and quasi-coherent finite space, $A=\OO(X)$. For any $p,q\in X$, the natural morphism $\OO_p\otimes_A\OO_q\to\OO_{pq}$ is an isomorphism.
\end{thm}

\begin{proof} Let $\delta\colon X\to X\times_A X$ be the diagonal morphism. It suffices to prove that the natural morphism $\OO_{X\times_AX}\to\delta_*\OO$ is an isomorphism. Since $X\times_A X$ is affine (by Corollary \ref{product-affine}) and $\delta_*\OO$ is quasi-coherent (because $X$ is quasi-coherent), it suffices to see that it is an isomorphism after taking global sections. By Corollary \ref{product-affine}, one has $\Gamma(X\times_AX,\OO_{X\times_AX})=A\otimes_A A=A$.
\end{proof}

\subsection{Schematic finite spaces. Semi-separated finite spaces.}

By definition, a schematic finite space is a   finite space such that $\RR\delta_*\OO$ is quasi-coherent. A stronger condition is the following:

\begin{defn} A finite space $X$ is {\it semi-separated} if $\delta\colon X\to X\times X$ is quasi-coherent and acyclic, i.e., $\delta_*\OO$ is quasi-coherent and $R^i\delta_*\OO=0$ for $i>0$.
\end{defn}

It is clear that any open subset of a schematic space is also schematic. We shall see that the converse is also true, i.e., being schematic is a local question. Any semi-separated space is obviously schematic. Moreover, we shall see that schematic spaces are locally semi-separated.

\begin{thm}\label{extension} {\rm (of extension)} Let $X$ be a schematic finite space. For any open subset $j\colon U\hookrightarrow X$ and any quasi-coherent module  $\Nc$ on $U$, $\RR j_*\Nc$ is quasi-coherent. In particular, any quasi-coherent module on $U$ is the restriction to $U$ of a quasi-coherent module on  $X$.
\end{thm}

\begin{proof} Let $\delta_U\colon U\to U\times X$ be the graphic of $j\colon U\hookrightarrow X$ and $\pi_U$, $\pi_X$ the projections of $ U\times X$ onto $U$ and $X$. By Theorem \ref{graphic}, one has an isomorphism $\LL \pi_U^*\Nc \overset\LL \otimes \RR{\delta_U}_*\OO_{\vert U} \overset\sim\to \RR{\delta_U}_*\Nc$. On the other hand, $\RR{\delta_U}_*\OO_{\vert U} = (\RR {\delta}_*\OO)_{\vert U\times X}$, which is quasi-coherent because $X$ is schematic. Hence, $\RR{\delta_U}_*\Nc$ is quasi-coherent. Since $\RR j_*\Nc = \RR {\pi_X}_*\RR{\delta_U}_*\Nc$, we conclude by Theorem \ref{qc-of-proj}.
\end{proof}

\begin{rem} The converse of the theorem also holds; even more: if $\RR j_*\OO_{\vert U}$ is quasi-coherent for any open subset $j\colon U\hookrightarrow X$, then $X$ is schematic.
\end{rem}

\begin{prop} A finite space $X$ is schematic if and only if $U_p$ is schematic  for any $p\in X$.
\end{prop}

\begin{proof} If $X$ is schematic, then $U_p$ is schematic. Conversely, assume that $U_p$ is schematic. We have to prove that $\RR\delta_*\OO$ is quasi-coherent. It suffices to see that $(\RR\delta_*\OO)_{\vert U_p\times U_q}$ is quasi-coherent for any $p,q\in X$. Let us denote $\delta'\colon U_{pq}\to U_{pq}\times U_{pq}$ the diagonal morphism of $U_{pq}$ and $i\colon U_{pq}\hookrightarrow U_p$, $j\colon U_{pq}\hookrightarrow U_q$ the inclusions.
Since $U_p$ and $U_q$ are schematic, $U_{pq}$ is schematic (i.e., $\RR\delta'_*\OO_{\vert U_{pq}}$ is quasi-coherent), and $\RR i$, $\RR j$ preserve quasi-coherence. By Proposition \ref{preservacion-cuasi}, $\RR (i\times j)_*\RR\delta'_*\OO_{\vert U_{pq}}$ is quasi-coherent, with $i\times j\colon U_{pq}\times U_{pq}\hookrightarrow U_{p}\times U_{q}$. That is, $(\RR\delta_*\OO)_{\vert U_p\times U_q}$ is quasi-coherent.
\end{proof}

\begin{prop}\label{schematic-affine} An affine finite space  is schematic if and only if it is semi-separated. Consequently, a finite space $X$ is schematic if and only if $U_p$ is semi-separated for any $p\in X$.
\end{prop}

\begin{proof} Let $X$ be an affine and schematic finite space. We have to prove that $R^i\delta_*\OO=0$ for $i>0$. Since  $X\times X$ is affine (Corollary \ref{product-affine}) and $R^i\delta_*\OO$ is quasi-coherent, it suffices to see that  $H^0(X\times X,R^i\delta_*\OO)=0$, $i>0$. Now,  $H^j(X\times X,R^i\delta_*\OO)=0$ for any $j>0$ and any $i\geq 0$, because $R^i\delta_*\OO$ is quasi-coherent and $X\times X$ is affine. Then $H^0(X\times X,R^i\delta_*\OO)=H^i(X,\OO)=0$ for $i>0$, because $X$ is affine.
\end{proof}

\begin{cor}\label{afinsemiseparado} Let $X$ be a schematic and affine finite space. An open subset $U$ of $X$ is affine if and only if it is acyclic.
\end{cor}

\begin{proof} Any affine open subset is acyclic by definition. Conversely, if $U$ is acyclic, it is enough to prove that any quasi-coherent module on  $U$ is generated by its global sections. This follows from the fact that any quasi-coherent module on $U$ extends to a quasi-coherent module on $X$ (by Theorem \ref{extension}) and $X$ is affine.
\end{proof}

\begin{thm} Let $X$ be a schematic finite space. For any quasi-coherent module   $\M$ on $X$, $\RR\delta_*\M$ is quasi-coherent.
\end{thm}

\begin{proof}  By Theorem \ref{graphic}, one has an isomorphism $\LL\pi_X^*\M\overset\LL \otimes \RR\delta_*\OO \simeq \RR\delta_*\M$. One concludes by the quasi-coherence of $\RR\delta_*\OO$.
\end{proof}

\begin{thm}\label{diagonal-semi-sep} Let $X$ be semi-separated. For any quasi-coherent module   $\M$ on $X$ and any $p,q\in X$, the natural morphism
\[ \M_p\otimes_{\OO_p}\OO_{pq}\to\M_{pq}\qquad (\M_{pq}=\Gamma(U_{pq},\M))\] is an isomorphism. That is, the natural morphism
\[ \pi ^*\M\otimes \delta_*\OO \to \delta_*\M\] is an isomorphism, where $\pi \colon X\times X\to X$ is any of the natural projections.  Moreover, $R^i\delta_*\M=0$ for any $i>0$. \end{thm}

\begin{proof} By Theorem \ref{graphic}, one has an isomorphism $\LL\pi^*\M\overset\LL \otimes \RR\delta_*\OO \simeq \RR\delta_*\M$. Now the result follows from the hypothesis $\RR\delta_*\OO =\delta_*\OO$.
\end{proof}

The following result justifies the name ``semi-separated'':

\begin{prop} A finite space $X$ is semi-separated if and only if $\delta\colon X\to X\times X$ is affine.
\end{prop}

\begin{proof} In both cases $\delta$ is quasi-coherent. Now, if $X$ is semi-separated, then  $U_{pq}$ is acyclic, hence affine by Corollary \ref{afinsemiseparado}, since it is contained in $U_p$, schematic and affine. Conversely, if $\delta $ is affine, then it is acyclic, so $X$ is semi-separated.
\end{proof}

\section{Schematic morphisms}

Let $f\colon X\to Y$ be a morphism and $\Gamma\colon X\to X\times Y$ its graphic. For each  $x\in X$ and $y\in Y$ we shall denote
\[ U_{xy}:=U_x\cap f^{-1}(U_y)=\Gamma^{-1}(U_x\times U_y),\quad \OO_{xy}:=\Gamma(U_{xy},\OO_X)=(\Gamma_*\OO_X)_{(x,y)}.\]


\begin{defn} We say that a morphism $f\colon X\to Y$ is  {\it schematic} if   $\RR \Gamma_*\OO_X$ is quasi-coherent. This means that
 for any $(x,y)\leq (x',y')$ and any $i\geq 0$, the natural morphism
 \[ H^i(U_{xy},\OO_X)\otimes_{\OO_{(x,y)}}\OO_{(x',y')}\to H^i(U_{x'y'},\OO_X)\] is an isomorphism.
\end{defn}

\begin{defn} We say that $f\colon X\to Y$ is  {\it locally acyclic} if the graphic $\Gamma$ is quasi-coherent and acyclic, i.e., $\Gamma_*\OO_X$ is quasi-coherent and $R^i\Gamma_*\OO_X=0$ for $i>0$. This last condition means that $U_{xy}$ is acyclic for any $x\in X$, $y\in Y$.
\end{defn}

Obviously, any locally acyclic morphism is schematic. Being schematic is local in $X$: $f\colon X\to Y$ is schematic if and only if $f_{\vert U_x}\colon U_x\to Y$ is schematic for any $x\in X$. If $f\colon X\to Y$ is schematic, then $f^{-1}(U_y)\to U_y$ is schematic for any  $y\in Y$; consequently, if $f\colon X\to Y$ is schematic, then $f\colon U_x\to U_{f(x)}$ is schematic for any $x\in X$. We shall see that the converse is also true if $Y$ is schematic.

\begin{ejems} \begin{enumerate} \item The identity $X\to X$ is schematic (resp. locally acyclic) if and only if $X$ is schematic (resp. semi-separated). A finite space $X$ is schematic (resp. semi-separated) if and only if for every open subset $U$, the inclusion $U\hookrightarrow X$ is an schematic (resp. locally acyclic) morphism.
\item If $Y$ is a punctual space, then any morphism $X\to Y$ is schematic and locally acyclic.
\item Let $f\colon S'\to S$ be a morphism between quasi-compact and quasi-separated schemes, and let $\U, \U'$ be locally affine coverings of $S$ and $S'$ such that $\U'$ is thinner than $f^{-1}(\U)$. Let $X'\to X$ the induced morphism  between the associated finite spaces. This morphism is schematic.
\end{enumerate}
\end{ejems}

In the topological case, one has the following result (whose proof is quite easy and it is omitted because it will not be used in the rest of the paper):

\begin{prop}  Let $f\colon X \to Y$ be a continuous map between finite topological spaces. The following conditions are equivalent.
\begin{enumerate}
\item $f$ is schematic.
\item $f$ is locally acyclic.
\item For any $x\in X$, $y\in Y$, $U_{xy}$ is  non-empty, connected and acyclic.
\end{enumerate}
Moreover, in any of these cases, $Y$ is irreducible (i.e., schematic)  and any generic point of $X$ maps to the generic point of $Y$.
If $X$ and $Y$ are irreducible, then $f$ is schematic if and only if the generic point of $X$ maps to the generic point of $Y$.
\end{prop}

\begin{thm}\label{sch-preserv-qc} Let $f\colon X\to Y$ be a schematic morphism and  $\Gamma\colon X\to X\times Y$ its graphic. For any quasi-coherent module  $\M$ on $X$, one has that $\RR\Gamma_*\M$ and $\RR f_*\M$ are quasi-coherent.
\end{thm}

\begin{proof} By Theorem \ref{graphic} one has an isomorphism $\LL\pi_X^*\M\overset\LL\otimes \RR\Gamma_*\OO_X\simeq \RR\Gamma_*\M$. Since $\RR\Gamma_*\OO_X$ is quasi-coherent, $\RR\Gamma_*\M$ is also quasi-coherent. Finally, if $\pi_Y\colon X\times Y\to Y$ is the natural projection, the isomorphism $\RR f_*\M \simeq \RR {\pi_Y}_*\RR \delta_*\M$ gives that $\RR f_*\M$ is quasi-coherent, by Theorem \ref{qc-of-proj}.
\end{proof}

\begin{thm}\label{sch-preserv-qc2} Let $X$ be a schematic finite space. A morphism $f\colon X\to Y$ is schematic if and only if $\RR f$ preserves quasi-coherence.
\end{thm}

\begin{proof}  The direct part is given by Theorem \ref{sch-preserv-qc}. For the converse, put the graphic of $f$, $\Gamma\colon X\to X\times Y$, as the composition of the diagonal morphism $\delta\colon X\to X\times X$ and $1\times f\colon X\times X\to X\times Y$. Now, $\RR \delta$ preserves quasi-coherence because $X$ is schematic and $\RR (1\times f)$ preserves quasi-coherence by the hypothesis and Proposition \ref{preservacion-cuasi}. We are done.
\end{proof}

\begin{prop}\label{composition-good} The composition of schematic morphisms is schematic.
\end{prop}

\begin{proof} Let $f\colon X\to Y$ and $g\colon Y\to Z$ be schematic morphisms,    $h\colon X\to Z$ its composition. The graphic  $\Gamma_h\colon X\to X\times Z$ is the composition of the morphisms
\[ X\overset {\Gamma_f}\to X\times Y \overset{1\times\Gamma_g}\to X\times Y\times Z \overset\pi \to X\times Z\] where $\pi\colon X\times Y\times Z \to X\times Z$ is the natural projection. $\RR\Gamma_f$ and $\RR\Gamma_g$ preserve quasi-coherence because $f$ and $g$ are schematic,  $\RR(1\times\Gamma_g)$ preserves quasi-coherence by Proposition \ref{preservacion-cuasi} and $\RR\pi$ preserves quasi-coherence by Theorem \ref{qc-of-proj}. Hence $\RR{\Gamma_h} $ preserves quasi-coherence.
\end{proof}

\begin{prop} The product of schematic morphisms is schematic. That is, if  $f\colon X\to Y$ and $f'\colon X'\to Y'$ are schematic morphisms, then  $f\times f'\colon X\times X'\to Y\times Y'$ is schematic.
\end{prop}

\begin{proof} The graphic of $f\times f'$  is the composition of the   morphisms $\Gamma_f\times 1\colon X\times X'\to X\times Y\times X'$ and $1\times\Gamma_{f'}\colon X\times Y\times X'\to  X\times Y\times X'\times Y'$. One concludes by Proposition \ref{preservacion-cuasi}.
\end{proof}

\begin{cor} The direct product of two schematic spaces is schematic.
\end{cor}

\begin{prop}\label{loc-acyclic-affine} Let $X$ and $Y$ be two affine finite spaces. A morphism $f\colon X\to Y$ is schematic if and only if it is locally acyclic.
\end{prop}

\begin{proof} It is completely analogous to the proof of Proposition \ref{schematic-affine}.
\end{proof}

\begin{cor} Let $Y$ be a schematic finite space. A morphism $f\colon X\to Y$ is schematic if and only if, for any $x\in X$, the morphism $f\colon U_x\to U_{f(x)}$ is locally acyclic.
\end{cor}

\begin{proof} If $f\colon X\to Y$ is schematic, then $f\colon U_x\to U_{f(x)}$ is schematic, hence locally acyclic by Proposition \ref{loc-acyclic-affine}. Conversely, assume that $f\colon U_x\to U_{f(x)}$ is locally acyclic for any $x\in X$. Since $Y$ is schematic, $U_{f(x)}\hookrightarrow Y$ is schematic, so the composition  $U_x\to U_{f(x)}\hookrightarrow Y$ is schematic, by Proposition \ref{composition-good}. Hence $f$ is schematic.
\end{proof}

\begin{thm}\label{graphofschematic} Let $f\colon X\to Y$ be a locally acyclic morphism and  $\Gamma\colon X\to X\times Y$ its graphic. For any quasi-coherent $\OO_X$-module  $\M$ and any $(x,y)$ in $X\times Y$, the natural morphism
\[ \M_x\otimes_{\OO_x}\OO_{xy}\to\M_{xy}\qquad (\M_{xy}=\Gamma(U_{xy},\M))\] is an isomorphism. In other words, the natural morphism
\[ \pi_X^*\M\otimes \Gamma_*\OO_X\to \Gamma_*\M\] is an  isomorphism, where $\pi_X\colon X\times Y\to X$ is the natural projection.  Moreover, $R^i\Gamma_*\M=0$ for $i>0$.
\end{thm}

\begin{proof} By Theorem \ref{graphic}, one has an isomorphism $\LL\pi_X^*\M\overset\LL\otimes \RR\Gamma_*\OO_X\simeq \RR\Gamma_*\M$. One concludes by the hypothesis $\RR\Gamma_*\OO_X =\Gamma_*\OO_X$.
\end{proof}

\begin{prop}\label{loc-acyc=Serre-afin}   A morphism $f\colon X\to Y$ is locally acyclic if and only if its graphic $\Gamma\colon X\to X\times Y$ is Serre-affine. If $X$ is schematic, then  $f$ is locally acyclic if and only if  $\Gamma$ is affine.
\end{prop}

\begin{proof} For the first part, in both cases   $\Gamma$ preserves quasi-coherence. If $\Gamma$ is Serre-affine, then $\Gamma$ is  acyclic, so $f$ is locally acyclic. Conversely, if $f$ is locally acyclic, for any quasi-coherent module  $\M$ on $X$, one has
$ R^i\Gamma_*\M =0$ for $i>0$, by Theorem \ref{graphofschematic}. Hence $\Gamma$ is Serre-affine.

Finally, assume that $X$ is schematic and  $f$ is locally acyclic. Then  $U_{xy}$ is acyclic and, by Corollary \ref{afinsemiseparado}, $U_{xy}$ is affine, since  $U_x$ is schematic and affine. Hence $\Gamma$ is affine.
\end{proof}

\begin{prop} The composition of locally acyclic morphisms is locally acyclic.
\end{prop}

\begin{proof} Let $f\colon X\to Y$ and $g\colon Y\to Z$ be locally acyclic morphisms,    $h\colon X\to Z$ its composition. By Proposition \ref{composition-good}, $h$ is schematic.  Hence ${\Gamma_h}_*\OO_X$ is quasi-coherent. To conclude we have to prove that $\Gamma_h$ is acyclic, i.e. for any $x\in X$, $z\in Z$, the open set $U_{xz}=U_x\cap h^{-1}(U_z)$ is acyclic. Let $y=f(x)$. It is clear that
\[ U_{xz}=  U_x\cap f^{-1}(U_{yz})\]
Let us denote  $\Gamma\colon U_x\to U_x\times U_y$ the graphic of the restriction of $f$ to  $U_x$, and let $\pi\colon U_x\times U_y\to U_y$ be the natural projection. Since $f$ is locally acyclic, $H^i(U_{xz},\OO_{U_x})= H^i(U_x\times U_{yz}, \Gamma_*\OO_{U_x})=  H^i(U_{yz}, \pi_* \Gamma_*\OO_{U_x})$, because $\pi$ has no higher direct images. Finally, $H^i(U_{yz}, \pi_*\Gamma_*\OO_{U_x})=0$ for $i>0$ by Corollary \ref{aciclico}, since $U_{yz}$ is an acyclic open subset of  $U_y$ (because $g$ is locally acyclic).
\end{proof}

\begin{prop} The product of locally acyclic morphisms is locally acyclic. That is, if  $f\colon X\to Y$ and $f'\colon X'\to Y'$ are locally acyclic morphisms, then  $f\times f'\colon X\times X'\to Y\times Y'$ is locally acyclic.
\end{prop}

\begin{proof} The graphic of $f\times f'$  is the composition of the   morphisms $\Gamma_f\times 1\colon X\times X'\to X\times Y\times X'$ and $1\times\Gamma_{f'}\colon X\times Y\times X'\to  X\times Y\times X'\times Y'$. These morphisms are Serre-affine by Propositions \ref{loc-acyc=Serre-afin} and \ref{fx1}, since  $f$ and $f'$ are locally acyclic. Then, its composition is  Serre-affine, i.e., $f\times f'$ is locally acyclic.
\end{proof}

\begin{cor} The direct product of two semi-separated spaces is semi-separated.
\end{cor}

\subsection{Affine schematic morphisms}

In this subsection we shall see that schematic affine morphisms share a lot of properties with affine morphisms of schemes.

\begin{thm}\label{Stein-fact}(Stein's factorization). Let $f\colon X\to Y$ be a schematic morphism. Then $f$ factors through an schematic morphism $f'\colon X\to Y'$ such that $f'_*\OO_X=\OO_{Y'}$ and an affine morphism $Y'\to Y$ which is the identity on the topological spaces. If $f$ is affine, then $f'$ is also affine and the functors
\[ \{ \text{Quasi-coherent }\OO_X-\text{modules}\} \overset{f'_*}{\underset{{f'}^*}{\overset\longrightarrow\leftarrow}}   \{ \text{Quasi-coherent }\OO_{Y'}-\text{modules}\} \]
are mutually inverse.
\end{thm}

\begin{proof} Take $Y'$ the finite space whose underlying topological space is $Y$ and whose sheaf of rings is $f_*\OO_X$. The morphism $f'\colon X\to Y'$ is the obvious one ($f'=f$ as continuous maps, and $\OO_{Y'}\to f_*\OO_X$ is the identity) and the affine morphism $Y'\to Y$ is that of example \ref{ejem-affine}. It is clear that $f'_*\OO_X=\OO_{Y'}$. Finally, let us see that $f'$ is schematic. Let $\Gamma'\colon X\to X\times Y'$ be the graphic of $f'$ (which coincides topologically with the graphic of $f$). Then, $\RR \Gamma'_*\OO$ is  quasi-coherent over $\OO_{X\times Y'}$ because it is quasi-coherent over $\OO_{X\times Y}$ ($f$ is schematic) and $\OO_{X\times Y'}$ is quasi-coherent over $\OO_{X\times Y}$. Finally, if $f$ is affine, then $f'$ is affine because $f^{-1}(U_y)=f'^{-1}(U_y)$. The last assertion of the theorem follows   Theorem \ref{projection-formula},  (3).
\end{proof}

\begin{rem} This factorization holds with  weaker hypothesis; it is enough $f$ to be quasi-coherent.
\end{rem}

\begin{prop}\label{affinemorphism} If $Y$ is an affine finite space and  $f\colon X\to Y$ is an affine schematic morphism, then $X$ is affine.
\end{prop}
\begin{proof}  Firstly, let us see that   $X$ is acyclic; indeed,  $H^i(X,\OO_X)=H^i(Y,f_*\OO_X)$ because $f$ is acyclic, and $H^i(Y,f_*\OO_X)=0$ because $Y$ is affine and $f_*\OO_X$ is quasi-coherent (because $f$  is schematic). Let us see that any quasi-coherent module  $\M$ on $X$ is generated by its global sections. Since $f$ is affine, the natural morphism $f^*f_*\M\to\M$ is surjective, so it suffices to see that  $f_*\M$ is generated by its global sections. This follows from the fact that  $f_*\M$ is quasi-coherent and $Y$ is affine.
\end{proof}

\begin{cor} The composition of two affine schematic morphisms is an affine schematic morphism.
\end{cor}

\begin{proof}  Let $f\colon X\to Y$ and $g\colon Y\to Z$ be affine and schematic morphisms. We already know that the composition is schematic. Let us see that it is also affine. For each $z\in Z$, we have to prove that  $(g\circ f)^{-1}(U_z)=f^{-1}(g^{-1}(U_z))$ is affine. Since $g$ is affine, $g^{-1}(U_z)$ is affine. One concludes by  Proposition \ref{affinemorphism}, because the morphism $f^{-1}(g^{-1}(U_z))\overset f\to g^{-1}(U_z)$ is schematic and affine.
\end{proof}

\begin{cor} Let   $f\colon X\to Y$ be a schematic morphism, with $X$ a schematic finite space and  $Y$    affine. Then $f$ is affine if and only if   $X$ is affine.
\end{cor}

\begin{proof} The direct statement is given by  Proposition \ref{affinemorphism}. Let us see the converse.  Since $f$ is schematic, it preserves quasi-coherence. It remains to prove that for any $y\in Y$, $f^{-1}(U_y)$ is affine. By Corollary \ref{afinsemiseparado}, it suffices to prove that $f^{-1}(U_y)$ is acyclic. Since $H^i(f^{-1}(U_y),\OO_X)=[ R^if_*\OO_X]_y$, we have to prove that $R^if_*\OO_X=0$ for $i>0$.

One has that $H^p(Y,R^qf_*\OO_X)=0$ for any $q$ and any $p>0$, because  $Y$ is affine and $R^qf_*\OO_X$ are quasi-coherent. Then
$H^0(Y,R^if_*\OO_X)=H^i(X,\OO_X)=0$, because $X$ is affine. Thus $ R^if_*\OO_X=0$, because $Y$ is affine and $R^if_*\OO_X$ is quasi-coherent.

\end{proof}

The following result is now immediate:

\begin{prop} Let $X$ be a schematic finite space and let $f\colon X\to Y$ be a schematic morphism. The following conditions are equivalent:
\begin{enumerate}
\item $f$ is an affine morphism.
\item For any affine open subset  $V$ of $Y$, the preimage  $f^{-1}(V)$ is affine.
\item There exists a covering of $Y$ by affine open subsets $V_i$ whose preimages by  $f$ are affine.
\item There exists a covering of   $Y$ by minimal open subsets whose preimages by   $f$ are affine.
\end{enumerate}
\end{prop}

Finally, let us see that, for schematic morphisms, the relative notion of affine morphism is indeed the relativization of the notion of affine space.

\begin{prop} Let $X$ be a schematic finite space and   $f\colon X\to Y$ a schematic morphism. The following conditions are equivalent:
\begin{enumerate}
\item $f$ is an affine morphism.
\item $f$ is acyclic and any quasi-coherent module on  $X$ is generated by its global sections on  $Y$.
\end{enumerate}
\end{prop}

\begin{proof} (1) $\Rightarrow$ (2) is a consequence of  Proposition \ref{Serre-Quasi-Afin}. For the converse, we may assume that  $Y$ is a minimal open subset, and then we have to prove that  $X$ is affine. Firstly, $X$ is acyclic because  $f$ is acyclic, $Y$ is affine and  $f$ preserves quasi-coherence. Secondly, any quasi-coherent module $\M$ on $X$ is generated by its global sections because   $f^*f_*\M\to\M$ is surjective and $f_*\M$ is generated by its global sections (because it is quasi-coherent and $Y$ is affine).
\end{proof}

\subsection{Fibered products}

We have seen how the flatness condition on a finite space yields good properties for quasi-coherent modules, which fail  for arbitrary ringed finite spaces. However, an important property is lost. While the category of arbitrary ringed finite spaces has fibered products, the subcategory of finite spaces has not. However, we shall see now that the category of schematic spaces and schematic morphisms has fibered products.

\begin{thm}\label{fiberedproduct} Let $f\colon X\to S$ and $g\colon Y\to S$ be schematic morphisms between schematic finite spaces. Then
\begin{enumerate}
\item The fibered product $X\times_S Y$ is a schematic finite space and the natural morphisms $X\times_SY\to X$, $X\times_SY\to Y$ are schematic.
\item If $f$ and $g$ are affine, then $h\colon X\times_S Y\to S$ is also affine and $h_*\OO_{X\times_SY}=f_*\OO_X\otimes_{\OO_S}g_*\OO_Y$.
\end{enumerate}
\end{thm}

\begin{proof} First of all, let us see that  $X\times_SY$ is a finite space, i.e., it has flat restrictions. Let $(x,y)\leq (x',y')$ in $X\times_S Y$. Let $s,s'$ be their images in $S$. We have to prove that $\OO_x\otimes_{\OO_s}\OO_y \to \OO_{x'}\otimes_{\OO_{s'}}\OO_{y'}$ is flat. Since $\OO_s\to\OO_{s'}$ is flat, the morphism $$\OO_x\otimes_{\OO_s}\OO_y \to (\OO_{x}\otimes_{\OO_{s}}\OO_{y})\otimes_{\OO_s}\OO_{s'}=(\OO_x\otimes_{\OO_s}\OO_{s'})\otimes_{\OO_{s'}} (\OO_y\otimes_{\OO_s}\OO_{s'})$$ is flat. Now, since $f\colon U_x\to U_s$ and $g\colon U_y\to U_s$ are schematic morphisms, one has that
$ \OO_x\otimes_{\OO_s}\OO_{s'} = \OO_{xs'}$ and $  \OO_y\otimes_{\OO_s}\OO_{s'}=\OO_{ys'}$, and then we have a flat morphism
\[ \OO_x\otimes_{\OO_s}\OO_y \to \OO_{xs'}\otimes_{\OO_{s'}}\OO_{ys'}.\]
Now, $\OO_{xs'}\otimes_{\OO_x}\OO_{x'}=\OO_{x's'}=\OO_{x'}$, and $\OO_{ys'}\otimes_{\OO_y}\OO_{y'}=\OO_{y's'}=\OO_{y'}$. Since $\OO_x\to\OO_{x'}$ and $\OO_y\to\OO_{y'}$ are flat, the morphism
\[ \OO_{xs'}\otimes_{\OO_{s'}}\OO_{ys'}\to \OO_{x'}\otimes_{\OO_x}(\OO_{xs'}\otimes_{\OO_{s'}}\OO_{ys'})\otimes_{\OO_y}\OO_{y'}=\OO_{x'}\otimes_{\OO_{s'}}\OO_{y'}\]
is flat. In conclusion, $\OO_x\otimes_{\OO_s}\OO_y \to \OO_{x'}\otimes_{\OO_{s'}}\OO_{y'}$ is flat.

Let us prove the rest of the theorem by induction on $\#(X\times Y)$. If $X$ and $Y$ are punctual, it is immediate. Assume the theorem holds for $\#(X\times Y)<n$, and let us assume now that $\# (X\times Y)=n$.

Let us denote $Z=X\times_SY$, and $\pi\colon Z\to X$, $\pi'\colon Z\to Y$ the natural morphisms. Whenever we take $z_i\in Z$, we shall denote by $x_i$, $y_i$ the image of $z_i$ in $X$ and $Y$.

(1) Assume that $X=U_x$, $Y=U_y$ and $f(x)=g(y)$. Let us denote $s=f(x)=g(y)$, $z=(x,y)\in Z$. Obviously $Z=U_z$. Let $\delta\colon Z\to Z\times Z$ be the diagonal morphism. Let us see that $R^i\delta_*\OO_Z = 0$ for $i>0$. Indeed, $(R^i\delta_*\OO)_{(z,z)}=0$ because $U_z$ is acyclic. Now, if $(z_1,z_2)\in Z\times Z$ is different from $(z,z)$, then, $U_{z_1z_2}=U_{x_1x_2}\times_{U_s}U_{y_1y_2}$, and $U_{x_1x_2}\times U_{y_2y_2}$ has smaller order than $U_x\times U_y$. Moreover, $U_{x_1x_2}$ and $U_{y_1y_2}$ are affine because $U_x$ and $U_y$ are schematic. By induction, $U_{z_1z_2}$ is affine, hence $(R^i\delta_*\OO)_{(z_1,z_2)}=0$. Let us see now that $\delta_*\OO$ is quasi-coherent. We have to prove that $\OO_{z_1}\otimes_{\OO_z}\OO_{z_2}=\OO_{z_1z_2}$ for any $z_1> z< z_2$. By induction, we have that $$\OO_{z_1z_2}=\OO_{x_1x_2}\otimes_{\OO_s}\OO_{y_1y_2},\quad \OO_{z_1}=\OO_{x_1}\otimes_{\OO_s}\OO_{y_1},\quad \OO_{z_2}=\OO_{x_2}\otimes_{\OO_s}\OO_{y_2}$$
One concludes easily because $\OO_{x_1x_2}=\OO_{x_1}\otimes_{\OO_s}\OO_{x_2}$ and $\OO_{y_1y_2}=\OO_{y_1}\otimes_{\OO_s}\OO_{y_2}$.
Hence $Z$ is schematic. Let us see that $\pi\colon Z\to X$ is schematic (hence affine). By Theorem \ref{sch-preserv-qc2} it suffices to see that $\RR \pi_*\M$ is quasi-coherent for any quasi-coherent module $\M$ on $Z$. Now, $R^i\pi_*\M=0$ for $i>0$, since $(R^i\pi_*\M)_x=0$ because $Z$ is acyclic and $(R^i\pi_*\M)_{x'}$ for $x'>x$ because $U_{x'}\times_{U_s}U_y$ is affine by induction. Finally, $\pi_*\M$ is quasi-coherent: indeed, $(\pi_*\M)_x\otimes_{\OO_x}\OO_{x'}=(\pi_*\M)_{x'}$ by Theorem \ref{aciclico}, because $U_{x'}\times_{U_s}U_y$ is, by induction, an acyclic open subset of $U_z$. Analogously, $\pi'\colon Z\to Y$ is schematic and affine. Finally, $\OO_Z(Z)=\OO_X(X)\otimes_{\OO_S(S)}\OO_Y(Y)$ because $\OO_z=\OO_x\otimes_{\OO_s}\OO_y$.

(2) Assume that $X=U_x$ (or $Y=U_y$). Let $s=f(x)$. Then $X\times_SY=U_x\times_{U_s}g^{-1}(U_s)$. If $g^{-1}(U_s)\neq Y$ we conclude by induction. So assume $Y=g^{-1}(U_s)$. If $Y=U_y$, we conclude by (1) if $g(y)=s$ or by induction if $g(y)\neq s$. If $Y$ has not a minimum,  for any $z\in Z$, $U_z$ is schematic by induction, hence $Z$ is schematic. Moreover $\pi'\colon Z\to Y$ is schematic and affine because it is schematic and affine on any $U_y$ by induction. Let us see now that $\pi\colon Z\to U_x$ is schematic. It suffices to see that $\RR \pi_*\M$ is quasi-coherent for any quasi-coherent module $\M$ on $Z$. We have to prove that
\[ (R^i\pi_*\M)_x\otimes_{\OO_x}\OO_{x'}\to (R^i\pi_*\M)_{x'}\] is an isomorphism for any $x'\in U_x$. By induction, for any $y\in Y$ one has \[ \M(U_x\times_{U_s}U_y)\otimes_{\OO_x}\OO_{x'}= \M(U_{x'}\times_{U_s}U_y).\]
Let us denote $Z'=U_{x'}\times_{U_s}U_y$, $\pi''$ the restriction of $\pi'$ to $Z'$ and $\M'$ the restriction of $\M$ to $Z'$; then,
 $\Gamma(Y, C^r\pi'_*\M)\otimes_{\OO_x}\OO_{x'} = \Gamma(Y, C^r\pi''_*\M')$. Now, since $\pi'$ and $\pi''$ are affine,
\[ \aligned H^i(U_x\times_{U_s}Y,\M)\otimes_{\OO_x}\OO_{x'}&= H^i(Y,\pi'_*\M)\otimes_{\OO_x}\OO_{x'} = H^i\Gamma (Y,C^\punto \pi'_*\M)\otimes_{\OO_x}\OO_{x'}\\ &= H^i\Gamma (Y,C^\punto \pi''_*\M')= H^i(Y,\pi''_*\M')=  H^i(U_{x'}\times_{U_s}Y,\M).\endaligned\]
Since $H^i(U_{x'}\times_{U_s}Y,\M)=(R^i\pi_*\M)_{x'}$, we have proved that $R^i\pi_*\M$ is quasi-coherent. If $Y$ is affine, then $U_x\times_{U_s}Y$ is affine because $U_x\times_{U_s}Y\to Y$ is schematic and affine, and $\pi'_*\OO_Z= g^*f_*\OO_X$, since this equality holds taking the stalk at any $y\in Y$. Taking global sections, one obtains $\OO_Z(Z)=\OO_x\otimes_{\OO_s}\OO_Y(Y)$.

(3) For general $X$ and $Y$. For any $z=(x,y)\in Z$, $U_z$ is schematic by (2), hence $Z$ is schematic. The morphisms $Z\to X$ and $Z\to Y$ are schematic because their are so on $U_x$ and $U_y$ respectively (by (2)). If $X$ and $Y$ are affine over $S$, then $Z$ is affine over $X$ and $Y$, because it is so over $U_x$ and $U_y$ respectively. Finally, if $X$ and $Y$ are affine over $S$, then $h_*\OO_Z=f_*\OO_X\otimes_{\OO_S}g_*\OO_Y$; indeed, the question is local on $S$, so we may assume that $S=U_s$. Hence $X$ and $Y$ are affine, and then $\pi'_*\OO_Z=g^*f_*\OO_X$, since this holds taking fibre at any $y\in Y$. Taking global sections, we obtain $\OO_Z(Z)=\OO_X(X)\otimes_{\OO_s}\OO_Y(Y)$.

\end{proof}

An easy consequence of this theorem is the following:

\begin{cor} Let $f\colon X\to S$ and $g\colon S'\to S$ be schematic morphisms between schematic spaces. If $f$ is affine, then $f'\colon X\times_SS'\to S'$ is affine. If in addition $f_*\OO_X=\OO_S$, then $f'_*\OO_{X\times_SS'}=\OO_{S'}$.
\end{cor}

\section{Embedding schemes into finite spaces}

For this section all finite spaces are assumed to be schematic.

\begin{defn} An affine schematic morphism $f\colon X\to Y$ such that $f_*\OO_X=\OO_Y$ is called a {\it qc-isomorphism}.
\end{defn}

A qc-isomorphism $f\colon X\to Y$ yields an equivalence between the categories of quasi-coherent modules on $X$ and $Y$, by Theorem \ref{Stein-fact}. A schematic morphism $f\colon X\to Y$ is a qc-isomorphism if and only if for any open subset $V$ of $Y$, $f\colon f^{-1}(V)\to V$ is a qc-isomorphism. If $f\colon X\to Y$ is a qc-isomorphism, then, for any schematic morphism $Y'\to Y$, the induced morphism $X\times_YY'\to Y'$ is a qc-isomorphism. Composition of qc-isomorphisms is a qc-isomorphism. If $X$ is affine and $A=\OO(X)$, the natural morphism $X\to \{ *,A\}$ is a qc-isomorphism.

\begin{ejem}\label{ejem-qc-iso} Let $S$ be a quasi-compact and quasi-separated scheme, $\U=\{ U_i\}$ and $\U'=\{ U'_j\}$ two (locally affine) finite coverings of $S$. Assume that  $\U'$ is thinner than $S$ (this means that, for any $s\in S$, $U'_s\subseteq U_s$). Let $\pi'\colon S\to X'$ and $\pi\colon S\to X$ be the finite spaces associated to $\U'$ and $\U$. Since $\U'$ is thinner that $\U$, one has a morphism $f\colon X'\to X$ such that $f\circ \pi'=\pi$. Then $f$ is a qc-isomorphism.
\end{ejem}

Let us denote by $\bold {Schematic}$ the category of schematic finite spaces and schematic morphisms and $\bold {SchFinSp_{qc-isom}}$ its localization  by qc-isomorphisms:

A morphism $X\to Y$ in $\bold {SchFinSp_{qc-isom}}$ is represented by a diagram $$\xymatrix{ & T\ar[dl]_\phi \ar[rd]^f & \\  X &  & Y  }$$
where $\phi$ is a qc-isomorphism and $f$ a schematic morphism. Two diagrams
$$\xymatrix{ & T\ar[dl]_\phi \ar[rd]^f & \\  X &  & Y  } \qquad \xymatrix{ & T'\ar[dl]_{\phi'} \ar[rd]^{f'} & \\  X &  & Y  }$$ are equivalent (i.e. they represent the same morphism in $\bold {SchFinSp_{qc-isom}}$) if there exist a schematic space $T''$ and qc-isomorphisms $\xi\colon T''\to T$, $\xi'\colon T''\to T'$ such that the diagram
\[ \xymatrix{ & & T''\ar[dl]_\xi \ar[dr]^{\xi'}\\ & T \ar[dl]_\phi \ar[drrr]^f  & & T' \ar[dlll]_{\phi'} \ar[dr]^{f'} &  \\ X & & & & Y
}\] is commutative. We denote by $f/\phi\colon X\to Y$ the morphism in $\bold {SchFinSp_{qc-isom}}$ represented by $f$ and $\phi$. The composition of two morphisms
\[\xymatrix{ & T \ar[dl]_\phi \ar[rd]^f & & T'\ar[dl]_\psi \ar[rd]^g & \\ X & & Y & & Z
}\] is given by $$\xymatrix{ & T\times_YT'\ar[dl]_\xi \ar[rd]^h & \\  X &  & Y  }$$ where $\xi$ (resp. $h$) is the composition of $\phi$ (resp. $g$) with the natural morphism $T\times_YT'\to T$ (resp.  the natural morphism $T\times_YT'\to T$). Notice that $T\times_YT'\to T$ is a qc-isomorphism because $\psi$ is a qc-isomorphism; hence $\xi$ is a qc-isomorphism.

Let us denote by $\bold{Schemes^{qc-qs}}$ the category of quasi-compact and quasi-separated schemes. We define a functor

$$\Phi\colon \bold{Schemes^{qc-qs}}\to \bold {SchFinSp_{qc-isom}}$$

\medskip
\noindent in the following way: Given a quasi-compact and quasi-separated scheme $S$, choose a (locally affine) finite covering $\U$ of $S$ and then define $\Phi(S)$ as the finite space associated to $S$ and $\U$. Let $f\colon S\to \overline S$ be a morphism of schemes, and let $\U$, $\overline \U$ be (locally affine) finite coverings of $S$ and $\overline S$, chosen to define $\Phi(S)$ and $\Phi(\overline S)$. Let $\U'$ be a (locally affine) finite covering of $S$ that is thinner that $\U$ and $f^{-1}(\overline\U)$. If $X'$ is the finite space associated to $S$ and $\U'$, one has natural morphisms
\[ g\colon X'\to \Phi(\overline S)\quad\text{and}\quad \phi\colon X'\to \Phi (S)\] and $\phi$ is a qc-isomorphism (Example \ref{ejem-qc-iso}). We define $\Phi(f)=g/\phi$.

\begin{prop}\label{S(affine)} Let $X$ be an affine and schematic finite space, $A=\OO(X)$. Then, the natural morphism of ringed spaces (see Examples \ref{ejemplos}, (5)) \[ \Sc (X)\to \Spec A\] is an isomorphism.
\end{prop}

\begin{proof} Let us denote $B=\underset{p\in X}\prod \OO_p$ and $A\to B$ the natural (injective) morphism. We know that  this morphism is faithfully flat (Proposition \ref{derivedcat-affine}). By faithfully flat descent,  we have an exact sequence (in the category of schemes)
\[ \Spec(B\otimes_AB) \aligned \,_{\longrightarrow} \\ \,^{\longrightarrow} \endaligned \,\Spec B\to \Spec A \] which is also an exact sequence in the category or ringed spaces (it is an easy exercise). Now, following the notations of Examples \ref{ejemplos} (5), one has $\Spec (B)=\underset{p\in X}\coprod \Spec\OO_p= \underset{p\in X}\coprod  S_p$, and $\Spec (B\otimes_AB)=\underset{p,q\in X}\coprod \Spec (\OO_p\otimes_A\OO_q) = \underset{p,q\in X}\coprod \Spec \OO_{pq}$, where the last equality is due to Theorem \ref{local-structure-affine}. That is, $\Spec A$ is obtained by gluing the schemes $S_p$ along the schemes $\Spec\OO_{pq}$.

Now, $U_{pq}$ is affine, because  $X$ is affine and schematic, hence semi-separated. Then, the morphism $\OO_{pq}\to\underset{t\in U_{pq}}\prod \OO_t$ is faithfully flat, so $ \underset{t\in U_{pq}}\coprod S_t \to \Spec\OO_{pq}$ is an epimorphism. We have then an exact sequence
\[ \underset{\underset{t\in U_{pq}} { \,_{p,q\in X}}}\coprod S_t  \aligned \,_{\longrightarrow} \\ \,^{\longrightarrow} \endaligned \, \underset{p\in X}\coprod S_p\to \Spec A\]

Notice that $S_t=S_{pt}=S_{qt}$ for any $t\in U_{pq}$. It is now clear that gluing the schemes $S_p$ along the schemes $S_{pq}=\Spec (\OO_{pq})$ (with arbitrary $p,q$) is the same as gluing the schemes $S_p$ along the schemes $S_{pq}$ (with $p\leq q$). This says that $\Spec A=\Sc (X)$.
\end{proof}

\begin{thm}\label{embedding-schemes} The functor  $\Phi\colon \bold{Schemes^{qc-qs}}\to \bold {SchFinSp_{qc-isom}}$   is fully faithful.
\end{thm}

\begin{proof} We have a functor $$\Sc\colon \bold {SchFinSp}\to \bold {RingedSpaces} $$ from schematic finite spaces to ringed spaces constructed in \ref{ejemplos} (5), such that $\Sc(\Phi (S))=S$ for any quasi-compact and quasi-separated scheme $S$. Let us see that $\Sc$ transforms qc-isomorphisms into isomorphisms. Let $f\colon X\to Y$ be a qc-isomorphism, and let us see that $\Sc (f)\colon \Sc (X)\to \Sc (Y)$ has an inverse $h\colon \Sc (Y)\to \Sc (X)$. For each $y\in Y$, let us denote $X_y=f^{-1}(U_y)$ and $f_y\colon X_y\to U_y$. We have that  $X_y$ is affine (because $f$ is affine)  and $\OO_{X_y}(X_y)=\OO_y$, because $f_*\OO_X=\OO_Y$. By Proposition \ref{S(affine)}, $\Sc (f_y)\colon \Sc (X_y)\to \Sc (U_y)$ is an isomorphism. Hence we have a morphism
\[ h_y\colon S_y=\Sc (U_y)\to \Sc (X)\] defined as the composition of
$\Sc (f_y)^{-1}\colon \Sc (U_y)\to \Sc (X_y)$ with the natural morphism $\Sc (X_y)\to \Sc(X)$. For any $y\leq y'$, let $r_{yy'}^*\colon S_{y'}\to S_y$ be the morphism induced by $r_{yy'}\colon \OO_y\to\OO_{y'}$. From the commutativity of the diagram

$$\xymatrix{ \Sc (X_{y'})\ar[r]\ar[d]  & \Sc (U_{y'})  \ar[d] \\ \Sc (X_{y})\ar[r]\ar[d]  & \Sc (U_y)  \\ \Sc (X) & }$$ it follows that  $h_y\circ r_{yy'}^*=h_{y'}$, hence one has a morphism $h\colon \Sc (Y)\to \Sc (X)$. It is clear that $h$ is an inverse of $\Sc (f)$. We have then proved that $\Sc$ transforms qc-isomorphisms into isomorphisms, hence it induces a functor
$$ \Sc\colon  \bold {SchFinSp_{qc-isom}} \to \bold {RingedSpaces}$$
and it is clear that $\Sc\circ\Phi$ is the identity on quasi-compact and quasi-separated schemes. To conclude, let us see that if $X$ is a $T_0$-schematic space such that $r_{pq}^*\colon \Spec\OO_q\to\Spec\OO_p$ is an open immersion for any $p\leq q$ (this is the case if $X=\Phi(S)$), then there is a natural qc-isomorphism $\Phi(\Sc(X))\to X$. Indeed, under this hypothesis, $S_p\to \Sc (X)$ is an open immersion for any $p\in X$, i.e. $\U=\{ S_p\}_{p\in X}$ is an open covering of $\Sc (X)$ and for any $p,q\in X$, $S_p\cap S_q=\underset{t\in U_{pq}}\cup S_t$. Let $\T_\U$  be the (finite) topology of $\Sc (X)$ generated by $\U$ and $\T_X$ the topology of $X$. Let us consider the map
\[  \begin{aligned}\psi\colon \T_X &\to \T_\U\\ U&\mapsto \psi(U)=\underset{t\in U}\cup S_t\end{aligned}   \]
It is clear that $\psi(U\cup V)=\psi(U)\cup \psi(V)$  and $\psi(U_p)=S_p$. Moreover, $\psi(U\cap V)= \psi(U)\cap \psi( V)$; indeed, since any $U$ is a union of $U_p's$ and $\psi$ preserves unions, we may assume that $U=U_p$ and $V=U_q$, and then the result follows from the equality  $S_p\cap S_q=\underset{t\in U_{pq}}\cup S_t$. In conclusion, $\psi$ is a  morphism of distributive lattices. It induces a continuous map $f\colon \Phi(\Sc (X))\to X$, whose composition with the morphism $ \pi\colon \Sc (X)\to \Phi(\Sc (X))$ gives a continuous map $\pi_X\colon \Sc (X)\to X$  such that $\pi_X^{-1}(U_p)=\psi(U_p)=S_p$. Since $\OO_p=\Gamma (S_p,\OO_{S_p})$,  one has a natural isomorphism $\OO_X\to f_*\OO_{\Phi(\Sc(X))}={\pi_X}_*\OO_{\Sc (X)}$, such that $f$ is a morphism of ringed spaces. Let us see that $f$ is a qc-isomorphism.  Let us see that $f$ is schematic: by Theorem \ref{sch-preserv-qc2}, it suffices to see that $\RR f_*\Nc$ is quasi-coherent for any $\Nc$ quasi-coherent module on $\Phi(\Sc(X))$. By Theorem \ref{schemes}, $\Nc=\pi_*\M$ for some quasi-coherent module $\M$ on $\Sc (X)$. Now, $\RR f_*\Nc=\RR {\pi_X}_*\M$ (because $R^i\pi_*\M=0$ for $i>0$) and $\RR {\pi_X}_*\M={\pi_X}_*\M$ because $\pi_X^{-1}(U_p)=S_p$ is an affine scheme. Finally, the quasi-coherence of ${\pi_X}_*\M$ is also a consequence of the equality $\pi_X^{-1}(U_p)=S_p$ and the hypothesis that $S_q\to S_p$ is an open immersion. Let us conclude that $f$ is a qc-isomorphism; $f$ is affine: $f^{-1}(U_p)$ is affine because $\pi^{-1}(f^{-1}(U_p))=S_p$ is affine. Since  $f_*\OO_{\Phi (\Sc (X))}=  \OO_X$, we are done.
\end{proof}

Let us denote by $\bold{AffSchSp}$ the faithful subcategory of $\bold{SchFinSp_{qc-isom}}$ whose objects are the affine schematic spaces. If $X\to Y$ is a qc-isomorphism and $Y$ is affine, then $X$ is affine. It follows that $\bold{AffSchSp}$ is the localization of the category of affine schematic spaces by qc-isomorphisms. Let us see that $\bold{AffSchSp}$ is equivalent to the category of affine schemes.

\begin{thm}\label{affinescheme-affinespace} The functor $$\aligned \Phi\colon \bold{AffineSchemes}&\to \bold{AffSchSp}\\ \Spec A&\mapsto (*,A)\endaligned$$ is an equivalence.
\end{thm}
\begin{proof} If $X$ is an affine schematic space and $A=\OO(X)$, then $X\to (*,A)$ is a qc-isomorphism and $\Sc(X)=\Spec A$ (Proposition \ref{S(affine)}). Now, for any affine schematic spaces $X$ and $Y$, with global functions $A$ and $B$, one has:
\[ \Hom_{\bold{AffSchSp}}(X,Y)=\Hom_{\bold{AffSchSp}}((*,A),(*,B)) \]
and it is clear that $\Hom_{\bold{AffSchSp}}((*,A),(*,B))=\Hom_{\text{rings}}(B,A)=\Hom_{\text{schemes}}(\Spec A,\Spec B)$.
One concludes that the functor $\bold{AffSchSp}\to \bold{AffineSchemes}$, $X\mapsto \Spec \OO(X)$, is an inverse of $\Phi$.
\end{proof}

\end{document}